\documentclass[12pt, reqno]{amsart}
\usepackage[top=3cm, left=2.5cm, right=2.5cm, bottom=2cm]{geometry}
\usepackage{amsthm,amsmath,amsfonts,amssymb, mathtools, dsfont}
\usepackage{graphicx}
\usepackage[mathcal]{euscript}
\usepackage{xcolor}
\usepackage{amsrefs}
\usepackage[colorlinks,citecolor=blue,urlcolor=blue, linkcolor=blue]{hyperref}

\theoremstyle{plain}
\newtheorem{theorem}{Theorem}[section]

\newtheorem{lemma}[theorem]{Lemma}
\newtheorem{proposition}[theorem]{Proposition}

\newtheorem{corollary}[theorem]{Corollary}

\theoremstyle{remark}
\newtheorem{definition}[theorem]{Definition}
\newtheorem{example}[theorem]{Example}
\newtheorem{remark}[theorem]{Remark}
\newtheorem{assumption}[]{Assumption}

\newcommand{\infdd}{\Rightarrow_{\text{fdd}}}
\newcommand{\mcY}{\mc{Y}}
\newcommand{\mcX}{\mc{X}}
\newcommand{\mchX}{\wh{\mc{X}}}
\newcommand{\mcH}{\mc{H}}
\newcommand{\mcP}{\mc{P}}

\newcommand{\mcN}{\mc{N}}

\newcommand{\td}{\mathbb{T}^d}
\newcommand{\zd}{\mathbb{Z}^d}
\newcommand{\cQ}{\mathcal{Q}}
\newcommand{\wh}{\widehat}

\newcommand{\Rd}{\mathbb{R}^d}
\newcommand{\stcomp}[1]{{#1}^{\mathsf{c}}}
\DeclareMathOperator{\fin}{Fin}

\DeclareMathOperator{\diam}{diam}

\DeclareMathOperator{\starop}{\star}
\newcommand{\bc}{\mathbf{c}}
\newcommand{\cI}{\mathcal I}
\newcommand{\seceq}{\setcounter{equation}{0}} 

\newcommand{\rev}[1]{#1}
\newcommand{\efe}[1]{#1}

\newcommand{\pr}{\mathbb P}
\newcommand{\ex}{\operatorname{\mathbb E}}
\DeclareMathOperator{\var}{var}
\DeclareMathOperator{\cov}{cov}
\newcommand{\der}{d}
\newcommand{\mc}{\mathcal}

\newcommand{\ind}{\operatorname{\mathds{1}}}

\newcommand{\bs}{\boldsymbol}
\newcommand{\nn}{\nonumber}

\begin{document}

\title{Central Limit Theorems for Local Functionals of Dynamic Point Processes}

\author{Efe Onaran}
\address{Department of Mathematics and School of Engineering and Applied Science, University of Pennsylvania}
\email{eonaran@seas.upenn.edu}
\thanks{
Significant parts of this work were carried out when both EO and OB were affiliated with the Technion -- Israel Institute of Technology.\\
EO was supported in part by the Israel Science Foundation, Grants 2539/17 and 1965/19.\\
OB was supported in part by the Israel Science Foundation Grant 1965/19 and in part by the EPSRC grant EP/Y008642/1.\\
 RJA was supported in part by the Israel Science Foundation, Grant  2539/17.}

\author{Omer Bobrowski}
\address{School of Mathematical Sciences, Queen Mary University of London}
\email{o.bobrowski@qmul.ac.uk}

\author{Robert J. Adler}
\address{Viterbi Faculty of Electrical and Computer Engineering, Technion--Israel Institute of Technology}
\email{radler@technion.ac.il}

\begin{abstract}
We establish \rev{finite-dimensional central} limit theorems for local, additive, interaction functions of temporally evolving point processes. The dynamics are those of a spatial Poisson process on the flat torus with points subject to a birth-death mechanism, and which move according to Brownian motion while alive. The results reveal the existence of a phase diagram describing at least three distinct structures for the limiting processes, depending on the extent of the local interactions and the speed of the Brownian motions. The proofs, which identify three different limits, rely heavily on Malliavin-Stein \rev{type CLTs for $U$-statistics} on a representation of the dynamic point process via a distributionally equivalent marked point process. 
\end{abstract}


\maketitle

\section{Introduction}
\seceq

Our interest lies in \rev{finite-dimensional} limit  theorems for {local} functionals defined on {dynamic} point processes, where the dynamics involve both {birth-death} and Brownian components.
 
More specifically, at time $t=0$ we are given a homogeneous Poisson point process of intensity $n$ on the $d$-dimensional flat torus $\td$. Each point has an independent, exponentially distributed lifetime, after which it is removed from $\td$. In addition, new points are added uniformly  on $\td$ according to a Poisson process 
with rate $n$. They too have independent, exponential lifetimes and are removed at death. This is the birth-death structure.

In addition, each point, while alive, moves  on $\td$ according  to  an independent Brownian motion with variance $\sigma^2$. 
We  denote the finite set of locations of all the  points alive at time $t\geq 0$ by  $\eta_n(t)\subset\td$.

We denote `interaction functionals' by $\xi_r$, where $\xi_r$ is a real valued function on finite subsets of $\td$ and the `locality parameter' $r>0$ plays the role of a `maximum interaction distance'. Somewhat more precisely, $\xi_r(A)\equiv 0$ if the diameter of $A$ is greater than $r$. A simple example of such a functional would count the number of cliques of fixed size, and of diameter no larger than  $r$, in a geometric graph. Our interest then would be in the time evolution of the number of cliques  in such a  graph when the nodes are the points in  $\eta_n(t)$. In particular, we are interested in the limit behavior of \rev{$U$-statistics} of a Poisson process with rate $n$ as $n\to \infty$.

It turns out that these limits depend on a delicate balance between the mean number of points, $n$, and the locality and motion rates. Given an $n$, let $r_n$ be the locality parameter, and $\sigma_n$ the speed of the Brownian motions. We will want both $r_n\to 0$ and $\sigma_n\to 0$ as $n\to \infty$ in order to obtain non-trivial limits. To save on notation, we will suppress the subscript on both $r_n$ and $\sigma_n$.

The delicate balance just mentioned leads to a phase diagram for the pair $(r,\sigma)$ for which we do not yet have a full description, but we can show that it contains at least three distinct regimes. 

 If $\sigma\ll r$, which we call the `slow regime', the  speed of the Brownian motions with respect to the maximum interaction distance is negligible, and we obtain the same limiting process for an appropriately normalized version of 
\begin{equation*}
     f_{n}(t) \coloneqq \sum_{\mcY \subseteq \eta_n(t)} \xi_r (\mcY)
\end{equation*} 
 as in \cite{onaran2022functional}, which studied the same model as here but without the Brownian motions. This limit is Gaussian, and is  representable as a weighted sum of independent Ornstein-Uhlenbeck processes with different parameters. 
 
 If $\sigma/ r\to c\in (0,\infty) $, which we call the  `moderate regime'  we prove that the corresponding limit is a special type of Gaussian process, 
 again representable via a sum, although in this case there is no simple analytic form for the covariances of the summed processes. \rev{In both the slow and moderate regimes, we find that the relative \emph{weight} of the summands depends on the asymptotic behavior of the density $nr^d$.}  
 
 Finally, in the `fast regime', in which $\sigma\gg r $, it turns out that 
 $f_n$, normalized, converges  to white noise, in the sense that its integral over time converges to Brownian motion. 
 \efe{Note that the decay rate of the covariance of the limit distribution is different in each regime. We conjecture that there exist ``transitional'' phases that fall between these three main regimes. A more complete description of the regime landscape is left for future work.}

The main tool we use in the proofs is the marked point process representation for the dynamic models introduced in \cite{onaran2022functional}. In addition, for finite dimensional convergence, we use results from the Stein normal approximation theory through Malliavin calculus that has been developing over the last decade \cite{nourdin2009,nourdinbook}. The normal approximation techniques we use differ from the ones we used in \cite{onaran2022functional} since the functional is no longer local in the strict sense (defined there) due to the movements of the particles. As we show, however, Malliavin-Stein theory is still applicable to our problem through its applications on $U$-statistics \cite{lach1,reitzner2013}, exhibiting yet another way this theory is useful in the study of dynamic point processes. Furthermore, the generality in the Malliavin-Stein theory for $U$-statistics allows us to extend our results in \cite{onaran2022functional} to sparser and denser regimes (determined by the choice of $r$) than the thermodynamic regime.

We conclude this section with some comments relating our model to another one which has been studied in some detail, 
and mentioning a practical use of the local functionals that we consider. 
Specifically,
our model is related to the `dynamic Boolean model' introduced in \cite{vandenberg}, where the authors  assume that a homogeneous Poisson point process in $\Rd$ is given at $t=0$ and, subsequently, each point moves according to a general continuous, stationary process. Concepts such as dynamic percolation, {coverage}, and {detection} were developed and studied in \cite{vandenberg} and in the many papers that were based on this model, including \cite{peresmobile} and \cite{stauffer}. Theoretical results have also found various applications, such as in mobile sensor networks \cite{liu2012dynamic}. The major difference between the model studied in the current paper and  the dynamic Boolean model is, of course, the inclusion of the birth-death dynamics. On a final note to the Introduction, we remark that a motivation to study the distribution of local interaction functionals of networks is that their distribution (also called `motifs' in the graph mining literature) serves as an important benchmark to detect anomalies (see \cite{chen2022algorithmic} and references).


The remainder of the paper is structured as follows: Section~\ref{notatsec} sets up the required notation and gives the main results of the paper. Section~\ref{prelimsec} relates the process we are studying to a particular marked point process that allows us to use results, e.g., Mecke's formula to write out formidable, but very useful, explicit expressions for various moments. The hard work is in Section~\ref{proofsec}, which, among other things, exploits these expressions to prove the results of Section~\ref{notatsec}.

\section{Notation and main results}\label{notatsec}
\seceq
Let $\td$ be a $d$-dimensional flat torus, taken as the  quotient space $\Rd/\mathbb{Z}^d$. Specifically, we will think of $\td$ as the cube 
$$
Q\coloneqq[-1/2,1/2)^d\subset \Rd,
$$ 
under the relation $0\sim 1$ \rev{coordinate-wise}, and endowed with the metric 
$$
\rho(x,y) \coloneqq \min_{\nu\in \mathbb{Z}^d} \|x-y+\nu \|, 
$$
for  $x,y\in Q$ and $\| \cdot \|$ the standard Euclidean metric.

Let $\fin(\td)$ denote the set of all finite subsets of $\td$. Then, as defined in the previous section, $\eta_n(t)$, the  set of points ``alive''  at time $t$, is an element of $\fin(\td)$. We are interested in the statistical properties of 
the additive function $f_{n}(t)$  defined on $\eta_n(t)$ by 
\begin{equation}\label{fntdef}
     f_{n}(t) \coloneqq \sum_{\mcY \subseteq \eta_n(t)} \xi_r (\mcY),
\end{equation}
where $\xi_r:\fin(\td)\to \mathbb{R}^+\coloneqq [0,\infty)$ satisfies $\xi_r(\mcY) = 0$ for all $|\mcY|\neq k$ for some $k\geq 2$. 
The value of $k$ will be fixed throughout this paper, and is therefore suppressed in the notation. In addition, $\xi_r$ is required to satisfy \rev{Assumptions~\ref{inv}-\ref{aecont}} for all  $0<r\leq 1$.

\begin{assumption}[Translation and scale invariance]\label{inv}
For any $\mcY\in \fin(\td)$ and $x\in\td$, define set translation and scalar multiplication in $\td$ by
\begin{align*}
    \mcY\oplus x &\coloneqq \{\pi( \pi^{-1}(y)+\pi^{-1}(x)): y\in \mcY \},\\
    \alpha \odot \mcY &\coloneqq \{\pi(\alpha \pi^{-1}(y)): y\in \mcY \},
\end{align*}
where $\pi:\Rd\to\td $ is the natural projection induced by the quotient operation and (with some abuse of notation) $\pi^{-1}: \td \to Q $ is the corresponding natural inverse. That is, $\pi^{-1}$ is the inverse of $\pi$ restricted to $Q$. For any $0<\alpha\leq 1$ we assume
\begin{align*}
    \xi_r(\mcY\oplus x) = \xi_r( \mcY),
\end{align*}
and 
\begin{align*}
    \xi_r(0,\mcY) =  \xi_{\alpha r}\left(0,\alpha \odot \mcY \right),
\end{align*}
where $\xi_r(0,\mcY)$ is a shorthand notation for $\xi_r(\{0\}\cup\mcY)$. 
\end{assumption}

Note that, throughout the paper, we will let $0$ denote both the origin in $\Rd$ and its projection on $\td$, since it will be clear from the context to which one we are referring.

\begin{assumption}[Localization]\label{localass}
There exists an $r$-independent constant $\delta\in (0,1/2)$ such that $\xi_r(\mcY) = 0$ if the diameter of the set  $ \mcY$ satisfies $\diam(\mcY)> \delta r$. 
\end{assumption}

\begin{assumption}[Boundedness]\label{boundedass}
The functional $\xi_r$ satisfies
\[
\|\xi\|_{\infty} \coloneqq \sup\limits_{\mcY\in \fin(\td)} \sup\limits_{0<r\leq 1} |\xi_r(\mcY)|<\infty.\]
\end{assumption}

\begin{assumption}[Feasibility] \label{feas} Let $\mathfrak{M}$ be  product Haar measure on $(\td)^{k-1}$. There exists a set $\mcH \subset (\td)^{k-1}$ satisfying $\mathfrak{M}(\mcH)>0$ such that
the functional $\xi_1$ satisfies,
    \begin{equation*}
     \xi_1(0, \bs x) >0 \quad\textrm{for all $\bs{x}\in \mcH$}.
    \end{equation*}
\end{assumption}

Note that here, and throughout the paper, we abuse notation somewhat, and allow $\xi_r$ to be applied to either finite subsets, or $k$-tuples. The latter will always be shown in bold face.



\begin{assumption}[Almost everywhere continuity]\label{aecont}
$\xi_1(0,\bs x)$ is $\mathfrak{M}$-almost everywhere continuous. 

\end{assumption}

A simple example should suffice to motivate both our assumptions and our results.

\begin{example}[Subgraph counts of geometric graphs]
 \label{example-graphs}
Let $G(\eta, r)$ be a geometric graph built over a finite set $\eta\subset \td$ with distance parameter $r$; i.e.\ an edge is placed between any two points in $\eta$ with distance less than or equal to $r$. Let $\mathcal{G}$ be a feasible geometric graph with $k$ vertices on $\td$; i.e.\  for iid uniform points $x_1,\ldots,x_k\in\td$, $\mathcal{G}$ satisfies 
\[\pr\left[\textrm{$G\left(\{x_1,\ldots,x_k\}, 1\right)$ is graph isomorphic to $\mathcal{G}$}\right]>0.\] 
Define
 $$\xi_r(\mcY) = \ind\{\textrm{$G(\mcY,r\delta/k)$ is graph isomorphic to $\mathcal{G}$}\}.$$
Then $\xi_r$ satisfies Assumptions \ref{inv}--\ref{aecont}, and ${f}_n(t)$ counts the number of subgraphs of $G(\eta_n(t),\allowbreak r\delta/k)$ graph isomorphic to $\mathcal{G}$. 
\end{example}


The \rev{finite-dimensional} limit theorems we prove in this paper will involve two different normalizations of the process $f_n(t)$, defined as
\begin{align}\label{barfndefeq}
\bar{f}_{n}(t) \coloneqq \frac{{f}_{n}(t) - \ex[{f}_{n}(t)] }{\sqrt{\text{var} [f_{n}(t)]} } 
\quad \text{and}\quad 
\tilde{f}_{n}(t) \coloneqq \frac{\bar{f}_{n}(t)}{\sqrt{2 M_n }} 
\end{align}
where
\[M_n\coloneqq \int_{0}^1\ex[\bar{f}_{n}(0) \bar{f}_{n}(t)]\der t.\]

In order to state our results succinctly, we use $\{\mc{U}_{j}(t):t\geq 0\}$, for some positive integer $j$, to denote the stationary, Gaussian, zero mean, Ornstein-Uhlenbeck (OU) process with  covariance function $$\text{cov}[\mc{U}_{j}(t_1), \mc{U}_{j}(t_2)] = \exp(-j|t_1-t_2|).$$
For a given sequence $\bc = (c_1,c_2,\ldots,c_k)$, \rev{$c_i\geq 0$ for all $1\leq i\leq k$,} we also define the following weighted superposition
\[
\mc{U}_{\bc}(t) := \sum_{j=1}^k c_j\mc{U}_j(t)
\]
  of \emph{independent} OU processes. We will denote the $\ell^2$ norm of the vector $\bc$ by $\|\bc\|$. 

Next, for a given positive integer $j$, 
let $\zeta_{j}:\mathbb R\to (0,1]$ be an even function decreasing in $[0,\infty)$ with $\zeta_j(0)=1$ and $\lim_{t\to\infty}\zeta_j(t)=0$, and let $\beta>0$. \rev{Also, let $\zeta_{j}$ be positive semi-definite, i.e., for all $\ell\geq 1$ and real numbers $\alpha_1,\ldots,\alpha_\ell$ the matrix $\mathrm{A} = \left(\zeta_j(\alpha_i-\alpha_m)\right)_{i,m=1}^\ell$.} is positive semi-definite.
We 
define $\mathcal{V}_{j}^\beta$ to be the zero mean, stationary Gaussian process with  covariance function

\begin{align}\label{premcVdef}
\text{cov}[\mc{V}_{j}^\beta(t_1), \mc{V}_{j}^\beta(t_2)] = \exp(-j|t_1-t_2|) \zeta_{j}(\beta|t_1-t_2|).
\end{align}
For a given $\bc = (c_1,c_2,\allowbreak\ldots, c_k)$, we define the  process
\begin{align}\label{mcVdef}
\mc{V}_{\bc}^\beta(t) := c_1\mc{U}_1(t) + \sum_{j=2}^k c_j\mc{V}_{j}^\beta(t),
\end{align}
 a weighted superposition of \emph{independent} processes $\{\mc{V}_j^\beta(t)\}_{j=2}^k$ and $\mc{U}_1(t)$. Note that for both the 
 $\mc{V}_{j}^\beta$ and $\mc{V}_{\bc}^\beta$, the functions $\zeta_j$ implicit in their definitions do not appear explicitly in the notation.
 
Throughout the paper, we use $f\lesssim g$ to denote that $f(n) = O(g(n))$, and $f\ll g$ to denote that $f(n) = o(g(n))$. In addition, $\asymp$ denotes same order; i.e.\ $\lim_{n\to\infty} f(n)/g(n)=C$ for some constant $C>0$. In particular, if $C=1$, then we denote this as $f(n)\approx g(n)$.


As already mentioned in the Introduction, for the following theorems and  throughout the paper we adopt the convention  that, despite the fact that  $r$ (the locality parameter) and $\sigma$ (the speed parameter) are actually  assumed to be functions of $n$, we drop the subscript $n$. Furthermore, we shall always hold the following assumption on the asymptotics of the parameters $n$, $r$, $\sigma$, which is needed for a central limit theorem in the static case \cite{penrosebook}.
\rev{\begin{assumption}[Asymptotics]\label{asymptoassum} We have
\[ 
\lim_{n\to\infty} r=\lim_{n\to\infty} \sigma = 0.
\]
and 
 \[
 n^kr^{d(k-1)} \to\infty.
 \]
\end{assumption} 
 }
 Finally, we shall use the notation $f_n \infdd f$ to denote the convergence of finite dimensional distributions of the stochastic processes $f_n$ to those of $f$.
 
 We can now state our main results. \rev{We remind the reader that our results hold under Assumptions~\ref{inv}-\ref{aecont} and \ref{asymptoassum}. The proofs which will follow through lemmas that establish covariance structure and finite dimensional distributions are given at the end of Section~\ref{proofssec}} 

\begin{theorem}[Slow regime]
\label{slow thm}
 If $\sigma  /r\to 0 $, then 
 \[
 \left\{\bar{f}_{n}(t):t\geq 0\right\}\  \infdd\  \left\{\mc{U}_{\bc}(t):t\geq 0\right\},
 \]
 for some $\bc$ with $\|\bc\|=1$.
 
 Furthermore, if $nr^d \to 0$, then $c_k=1$. If $ nr^d\to\infty$, then $c_1=1$. If $ nr^d\to\gamma\in (0,\infty)$, then $c_j>0$ for all $1\leq j\leq k$.  
\end{theorem}

\begin{theorem}[Moderate regime] 
\label{ModerateThm}
 If $\sigma /r \to \sqrt{\beta}\in(0,\infty)$, then
 \[
 \left\{\bar{f}_{n}(t):t\geq 0\right\} \ \infdd  \ 
 \left\{\mathcal{V}_{\bc}^\beta(t):t\geq 0\right\},
 \]
 for some $\bc$ with $\|\bc\|=1$, and some functions $\zeta_j$ (see Eq.~\eqref{premcVdef}). 
 
 The characterization of the entries of $\bc$ with respect to the asymptotics of $nr^d$ given in Theorem~\ref{slow thm} also holds here. 
\end{theorem} 

Recall that the definition of the process $\mathcal{V}_{\bc}^\beta$ 
involves a collection of  functions $\zeta_j$  as described prior to \eqref{premcVdef}, and the existence of these functions, along with the properties listed there, is implicit in the above theorem. They are dependent on moment properties of the interaction functionals, and are defined in Section~\ref{proofsec}, at \eqref{definezeta}.

\begin{theorem}[Fast regime]
\label{white noise thm}
 If \[ 
 1\ll\sigma /r \ll (n^kr^{d(k-1)})^{\frac{1}{4}-\epsilon},
 \] 
 for some $0<\epsilon<\frac{1}{4}$, and if $nr^d\to 0$, $n\sigma^d \lesssim 1 $, and $d(k-1)\geq 3$, then
 \[
 \left\{\int_0^t\tilde{f}_{n}(s)\der s:t\geq 0\right \} \ \infdd \ 
 \left\{B(t):t\geq 0\right\},
 \]
where $B$ is a standard Brownian motion. 
\end{theorem} 
\rev{
\begin{remark}
Note that the first and second moments of $f_n(t)$, which are used to define $\bar{f}_n(t)$ \eqref{barfndefeq}, do not depend on the speed parameter $\sigma$, due to stationarity. Specifically, we will later show in the proofs that
 $\ex[f_n(t)] \asymp n^kr^{d(k-1)}$, and  
\begin{alignat*}{4}
\var[f_n(t)] \asymp \begin{cases}  n^kr^{d(k-1)} \quad &\text{ if $nr^d\to 0$},\\
     n \quad &\text{ if $nr^d\to \gamma \in (0,\infty)$, and } \\
    n^{2k-1}r^{d(2k-2)} \quad &\text{ if $nr^d\to \infty$}.
    \end{cases}
\end{alignat*}
\end{remark}
}

Informally, three theorems above  characterize the relative effects of the birth-death and Brownian  dynamics on the limiting local functional. In the slow regime, Theorem~\ref{slow thm} states that the impact of the motion is {negligible}, consistent with the results of \cite{onaran2022functional}. In the moderate regime,
where $\sigma$ and $r$ are comparable, both the birth-death  and the motion of the points impact on the limit. The pure OU component \eqref{mcVdef} in the superposition defining the limit process in this regime comes from  correlations between the total numbers of points in the system across different time instances, therefore playing the same role  as it did in the slow regime. The other components, reflecting the effect of the Brownian motions, have a covariance function with faster decay than that of the pure OU component.   

In the fast regime, the influence of  the birth-death component weakens in all the components but the first. To obtain a meaningful convergence result in this regime, however, we need to assume that  the locality is strong,  in the sense that $nr^d\to 0$. With the correct normalization, this weakens the correlation over time, resulting in a `white noise' limit, which is indicated by its time-integrated version converging in finite dimensional distributions to the Brownian motion.  It is plausible that the additional asymptotic upper bounds on $\sigma /r$ and $n\sigma^d$ in the conditions of Theorem~\ref{white noise thm} are not essential, but rather an artifact of our proof techniques. The asymptotics for the normalization, $M_n$ \eqref{barfndefeq}, necessary for the proof of Theorem~\ref{white noise thm}, will be presented in Proposition~\ref{propintasymp}.

Finally, we note that, as opposed to the treatment in \cite{onaran2022functional}, our convergence results here are not in terms of full weak convergence in the Skorokhod space. The missing component is tightness, which it would be natural to handle via a fourth moment argument. \rev{ In \cite{onaran2022functional} 
fourth moment calculations were helped by the fact that the lack of motion by the particles meant values of the functional on disjoint collections of particles were independent, regardless of time.  This is no longer true when the particles are allowed to move. (See the equation below A.4 in \cite{onaran2022functional}, for which we do not have a simple analogue.)} Consequently, computing these moments
turned out to be a computationally challenging, not to mention exceedingly tedious, task,
which does not seem to justify the considerable effort and space it would require in this
paper. 

\section{Preliminaries}\label{prelimsec}
\seceq
In this section we introduce the main tools that we will use in our proofs.

\subsection{Marked process model}
Here we describe a marked Poisson process model that is stochastically equivalent to $\eta_n(t)$ on any predetermined time interval $[0,T]$. This will make the notation easier and the proofs more intuitive. It will also let us use the normal approximation theory of \emph{U}-statistics in \cite{lach1,lach2,reitzner2013} by representing the dynamic model in a Poisson space. In this marked model, which we will denote as $\wh \eta_{n,T}$ for some $T>0$, we are given a homogeneous Poisson point process on the unit cube $Q$ with the rate $n(1+T)$. Each point, $x$, in the configuration is also given three independent marks $(B_x,L_x,Z_x)$. The \emph{birth-time} $B_x$ is defined as
\[
    B_x = Y_xU_x,    
\]
where $Y_x,U_x$ are independent random variables, 
\[\pr[Y_x=0] = 1-\pr[Y_x = 1] = \frac{1}{1+ T},\]
and $U_x$ is uniformly distributed on $[0,T]$.
The \emph{lifetime} $L_x$ follows an exponential distribution with unit mean. Finally, the \emph{path} mark $Z_x\in C_{\Rd}[0,T]$ is an independent Brownian motion in $\Rd$ with variance $\sigma^2$ in each dimension, satisfying $Z_x(0)=0$. Here $C_{\Rd}[0,T]$ denotes the set of continuous functions $[0,T]\to\Rd$. Note that, equivalently, $\wh \eta_{n,T}$ can be described as a Poisson point process on the product space $Q\times\mathbb R_+ \times \mathbb R_+ \times C_{\Rd}[0,T]$ due to the \emph{Marking theorem} \cite[Theorem 5.6]{poissonlectures}. 

Next, for  each point $\wh{x} \coloneqq \left(x,(B_x,L_x,Z_x)\right)$ of the process $\wh\eta_{n,T}$, define
\begin{align}
    \tau_t(\wh{x}) :=& \ind\{B_x \le t < B_x+L_x\},\label{tautxdef}\\
    \tau_t({\mchX}) :=& \prod_{\wh{x}\in\mchX}\tau_t({\wh{x}}), \quad t\geq0.\nn
\end{align} 
Thus $\tau_t(\wh{x})$ is the indicator function registering whether or not $\wh{x}$ is alive at time $t$, and $\tau_t({\mchX})$  registers whether or not  all points in the collection $\mchX$ are alive at $t$.  We also use 
\begin{align}\label{pimchxdef}
    \mchX (t)\coloneqq \left\{x+Z_x(t)-Z_x(B_x): \wh x\in\mchX\right\},  
\end{align}
to denote the locations of each point in $\mchX$ at time $t$, and the shorthand notation \begin{align}\label{whxidef}
    \wh \xi_r(\mchX (t))\coloneqq \xi_r(\pi(\mchX (t))),
\end{align} 
where the projection $\pi$ acts on $\mchX (t)$ element-wise. 
With this notation, the random process that is
of core interest to us is therefore $\wh\eta_{n,T}$, and the main candidate for our limit theorems is 
\begin{align}\label{fnhatdef}
    \widehat f_n(t)\coloneqq \sum_{\mchX\subseteq \wh \eta_{n,T}}  \wh \xi_r(\mchX (t))  \tau_t(\mchX).
\end{align}
 
\begin{remark}
The following  properties of $\wh \xi_r$,  which play a major role  in our proofs, follow from the  natural assumptions on $\xi_r$.
\normalfont
\begin{description}
    \item[Translation and Scale Invariance.] For any $\mcX\in \fin(\Rd)$, $x\in\Rd$, and $0< r\leq 1$,
    \begin{align}\label{transinvxihat}
       \wh\xi_r(\mcX + x) = \xi_r\left(\pi(\mcX + x)\right) = \xi_r\left(\pi(\mcX) \oplus \pi(x)\right) = \xi_r\left(\pi(\mcX) \right) = \wh\xi_r(\mcX )
    \end{align}
Furthermore, if $\mcX\in \fin(Q)$,
\begin{align}\label{scaleinvxihat}
    \wh\xi_1(0,\mcX) &=  \xi_{1}\left(0,\pi(\mcX) \right) = \xi_{r}\left(0,\pi\left(r \pi^{-1}\left( \pi(\mcX)\right)\right) \right)
    = \xi_{r}\left(0,\pi\left(r \mcX\right) \right)= \wh \xi_{r}\left(0,r \mcX\right).
\end{align}

\item[Feasibility.] \label{feasxihat}
Assumption~\ref{feas} implies the existence of a nonempty set $\pi^{-1}(\mcH)\subset Q^{k-1}$ with positive Lebesgue measure, for which $\wh \xi_1(0, \bs y)>0$ for all $y\in \pi^{-1}(\mcH)$, a fact that we will use often in our proofs. With $\wh \xi_1(0, \bs y)>0$, we abuse the notation between $k$-tuples and the sets, as explained in Assumption~\ref{feas}.  

\item[Locality.] \label{localxihat}
Due to Assumption~\ref{localass}, for all $0<r\leq1$, there exists 
a set $\mathcal{Z}\subset Q^{k-1}$ with positive Lebesgue measure such that, 
\begin{equation*}
     \wh \xi_r(0, \bs y) =0 \quad\textrm{for all $\bs{y}\in \mathcal{Z}$.}
\end{equation*}
In particular, $\mathcal{Z}\supseteq Q\setminus B_{\delta r}(0)$, where 
\[B_{\delta r}(\nu)\coloneqq \{x\in \Rd: \|x-\nu\|\leq \delta r \}, \quad \nu\in \zd.\]

\item[Continuity.]\label{aecontremark}
Due to Assumption~\ref{aecont}, $\wh \xi_1(0,\bs y)$ is Lebesgue almost everywhere continuous on $Q^{k-1}$.
\end{description}
\end{remark}

\begin{remark}\label{stationspacexihat}
Two identities related to the marked process construction given below will prove essential in our proofs. \efe{Note that they both rely on the fact that we work on the torus, a compact space.
}

\normalfont
\begin{description}

\item[Spatial homogeneity.]

Let $X$ be uniformly distributed in the cube $Q = [-1/2,1/2)^d\subset\Rd$ and let $Y\in\Rd$ be a random vector  independent of $X$. Then, $\pi(X)$ has the same distribution as $\pi(X+Y)$.

\item[Stationarity in time.]
Due to the spatial homogeneity, independence of the marks of each point, and the Markov property of Brownian motion and exponential lifetimes, joint distribution of $\wh f_n(t_1)$ and $\wh f_n(t_2)$ is the same as that of $\wh f_n(0)$ and $\wh f_n(t_2-t_1)$ for all $0\leq t_1\leq t_2\leq T $.
\end{description}
\end{remark}

The next observation that follows just from the construction of the marked point process will enable studying the distribution of $f_n(t)$ via that of $\wh f_n(t)$. 
\begin{proposition}\label{Propsamedist}
For any $T>0$,
    $\{\widehat f_n(t): 0\leq t\leq T\}$ and $\{f_n(t): 0\leq t\leq T \}$ have the same finite dimensional distributions.
\end{proposition}
\begin{proof}
For $m\geq 1$, take $0\leq t_1<t_2<\ldots< t_m\leq T$ and $\omega_1,\ldots, \omega_m\in \mathbb{R}_+$, Define
    $\eta_{n,T}$ to be the subset of $\td$ composed of the initial locations of the points of $\eta_n(t)$ that are born in the interval $0< t\leq T$.
Note that, $\eta_{n, T}$ is a Poisson point process
on $\td$ with rate $nT$, and  therefore $\eta_{n, T} \cup \eta_n(0)$ is also Poisson process on $\td$, but with rate $n(1+T)$. (c.f.\ The {superposition theorem} for Poisson processes, e.g.\ Theorem 3.3 in \cite{poissonlectures}). From \eqref{fntdef} we can write 
\begin{equation}\label{fntobs33}
     f_{n}(t) \coloneqq \sum_{\mcY \subseteq \eta_n(0)\cup \eta_{n, T}} \xi_r (\mcY(t))\prod_{y \in \mcY}\ind\{\textrm{$y$ is alive at $t$}\}.
\end{equation}
where $\mcY(t)$ is a shorthand notation for the locations of points of $\mcY$ at time $t$. 
Note that $ \mcY(t)$ and $\ind\{\textrm{$y$ is alive at $t$}\}$ are independent, and due to the construction of $\wh\eta_{n,T}$, they have the same distributions as $\pi(\wh \mcX(t))$ and $\tau_t(\wh x)$ given in \eqref{pimchxdef} and \eqref{tautxdef}, respectively. 
Therefore, comparing \eqref{fnhatdef} and \eqref{fntobs33}, the following holds
\[
    \ex\left[\exp\Big(-\sum_{i=1}^m \omega_i f_n(t_i)\Big)\Big|\, \eta_{n, T} \cup \eta_n(0) \right] =
    \ex\left[\exp\Big(-\sum_{i=1}^m \omega_i \wh f_n(t_i)\Big)\Big|\,  \wh\eta_{n,T} \right]. 
\]
Equivalence of the joint distributions $f_n(t_1),\ldots, f_n(t_m)$ and $\wh f_n(t_1),\ldots,\wh  f_n(t_m)$ follows from the fact that $\eta_{n, T} \cup \eta_n(0)$ and $\wh\eta_{n,T}$ have the same distributions, and that the distributions of random vectors are determined through their Laplace transforms (see Proposition B.4 in \cite{poissonlectures}).
\end{proof}

\subsection{Counting lemma}
The following  lemma  generalizes common counting techniques for Poisson processes, when the counted objects assume a specific structure of intersections. The following notation is needed for its statement. 

\begin{definition}\label{defn:pattern}
Let $\cI_\ell$ denote a collection of natural numbers indexed by the nonempty subsets of $[\ell] = \{1,\ldots,\ell\}$, i.e. $\cI_\ell = (I_J)_{J\subset [\ell], J\ne\emptyset}$. Given a sequence of finite subsets $\mcX_{1},\ldots,\mcX_{\ell}$ of a set $\cQ$, suppose that for all nonempty $J\subset[\ell]$ we have
\[
\left| \left(\bigcap_{j\in J} \mcX_j \right)\cap\left(\bigcap_{j\not\in J} \stcomp\mcX_j \right)\right| = I_J,
\]
where $\stcomp\mcX_j = \cQ\setminus \mcX_j$.
In this case we say that $\mcX_{1},\ldots,\mcX_{\ell} $ obey the \emph{intersection pattern} $\cI_\ell$, and denote this by
$(\mcX_{1},\ldots,\mcX_{\ell}) \in \cI_\ell$.
\end{definition}

In what follows, we will typically  write $\cI$ for $\cI_\ell$, unless $\ell$ is explicitly required. Set 
\[|\cI| := \sum_{J} I_{J}  = |\mcX_1\cup\cdots\cup \mcX_{\ell}|.\]
Fixing $\cI$, and given a tuple of points ${\bs x} = (x_1,\ldots,x_{|\cI|})$ in $\cQ$, 
let
\[
\Psi_{\cI}(\bs x) \coloneqq (\mcX_1,\ldots,\mcX_\ell) 
\]
be a splitting of the tuple $\bs x$ into  $(\mcX_1,\ldots,\mcX_\ell)\in \cI$, in an arbitrary but fixed manner. 
\efe{
\begin{example}
    Take $\cQ = \mathbb{Z}$, $\ell=3$, and $\mcX_1=\{1,2\}$, $\mcX_2=\{2,3\}$, $\mcX_3=\{2,3,4,5\}$. Then, $(\mcX_1, \mcX_2, \mcX_3)\in \cI$, $\cI = (I_{\{1\}}, I_{\{2\}}, I_{\{3\}}, I_{\{1,2\}}, I_{\{2,3\}}, I_{\{1,3\}}, I_{\{1,2,3\}})$ with $I_{\{1\}} = I_{\{2,3\}} = 1$, $I_{\{2\}} = I_{\{1,2\}} = I_{\{1,3\}}= 0$, $I_{\{3\}} = 2$, $I_{\{1,2,3\}} = 1$, and $|\cI| = 5$. On the other hand, if we were given $\bs x = (1,2,3,4,5)$, a $\Psi$, the exact definition of which is omitted to save space, can be found which would  satisfy  $\Psi_{\cI}(\bs x) = (\mcX_1,\mcX_2,\mcX_3)$. 
\end{example}}

 The following statement is a generalization of the well-known Mecke's formula for Poisson point processes. 

\begin{lemma}\label{Palmhigh}
For any $n>0$, let $\mcP_{n}$ be a Poisson point process on the point space $\cQ$ with intensity measure $n\mu$, where $\mu$ is a probability measure on $\cQ$. Let $h( \bar{\mcX}) $ be a bounded measurable real function with $\bar{\mcX} \coloneqq(\mcX_{1},\ldots,\mcX_{\ell})$ and $\mathcal{I}$ an intersection pattern. Then,
\[
     \ex\Big[\sum_{\mcX_{1} \subseteq \mcP_{n}} \cdots \sum_{\mcX_{\ell} \subseteq \mcP_{n}} h(\bar{\mcX})\ind\{ \bar{\mcX} \in \cI\}\Big]  =  \frac{n^{|\cI|} \ex \big[h\big(\Psi_{\cI}(\bs X)\big) \big] }{\prod_{J}I_{J}!},   
\]
where $\bs X$ is a tuple of $|\cI|$ iid points in $\cQ$, with distribution $\mu$.  
\end{lemma}
\begin{proof} See the proof of Lemma~3.4 in  \cite{onaran2022functional} for the special case $\cQ = \Rd$, which can be generalized in a straightforward manner.
\end{proof}

\section{Proofs}\label{proofsec}
\seceq
\subsection{Covariance characterization}
As a first step towards proving  the finite dimensional weak convergence of $\bar f_n(t)$ and $\int_0^t\tilde{f}_{n}(s)\der s$, we  derive expressions for their limiting covariance functions. 
The lemmas below characterize the limit of the covariance function $\ex[\bar f_n(t) \bar f_n(t+\Delta)]$ in different regimes of $r$ and $\sigma$ with respect to $n$, where $\Delta$ may also depend on $n$, but $\Delta\lesssim 1$. Note that for the proofs of Theorems~\ref{slow thm} and \ref{ModerateThm} we need only the special cases of these lemmas when $\Delta>0$ is a constant. We state them in the more generality since we will need them where $\Delta$ is also allowed to change with $n$ for the proof of Theorem~\ref{white noise thm}. 

\begin{lemma}\label{covOUlemma}
If $\sigma\sqrt{\Delta}/r \to 0$ then there exist non-negative constants $\lambda_1,\ldots,\lambda_k$, with $\sum_{j=1}^k\lambda_j=1$, such that, for all $t\geq 0$,
\begin{equation*}
     \ex[\bar f_n(t) \bar f_n(t+\Delta)] \approx \sum_{j=1}^k \lambda_j \exp(-j\Delta). 
\end{equation*}
\rev{If $\lim_{n\to\infty} nr^d =0$ then $\lambda_k=1$. If  $\lim nr^d \in(0,\infty)$ then the constants $\lambda_1,\ldots, \lambda_k$ are strictly positive and depend on the functional $\xi_1$ (see \eqref{lambdaexact}). Finally, if $\lim_{n\to\infty} nr^d =\infty$ then $\lambda_1=1$.}
\end{lemma}

\begin{lemma}\label{covModlemma}
If $\sigma\sqrt{\Delta}/r \to \sqrt{\beta}\in(0,\infty)$ then there exists a set of strictly positive, decreasing functions $\zeta_j:[0,\infty)\to (0,1]$, $\zeta_j(0)=1$, $2\leq j\leq k$, \rev{specified in \eqref{definezeta},} such that, for all $t\geq 0$,
\[
     \ex[\bar f_n(t) \bar f_n(t+\Delta)] \approx \lambda_1 \exp(-\Delta) + \sum_{j=2}^k \lambda_j \exp(-j\Delta)\zeta_j(\beta). 
\]
The functions $\zeta_2,\ldots, \zeta_k$ depend on $\xi_1$, and the constants $\lambda_1,\ldots, \lambda_k$ are the same as those in Lemma~\ref{covOUlemma}.
\end{lemma}
\rev{
\begin{remark}
In our proof of the above lemma, we will find that there are concise and probabilistically interpretable expressions for the functions $\zeta_j$. See \eqref{definezeta} and \eqref{zetajprob}.   
\end{remark}}

\begin{lemma}\label{covNoiselemmabar}
If $\sigma\sqrt{\Delta}/r \to \infty$, then, for all $t\geq 0$,
\[
    \ex[\bar{f}_n(t), \bar{f}_n(t+\Delta)] \approx \frac {\sum\limits_{j=1}^{k} \exp(-j\Delta)  \frac{ \kappa_1}{j!((k-j)!)^2} (2\pi\Delta)^{-d(j-1)/2}j^{-d/2} (n\sigma^d)^{-j} \left(\frac{\sigma}{r}\right)^d} 
    {\sum\limits_{j=1}^{k}\frac{ \kappa_{j}}{j!((k-j)!)^2 }   (nr^d)^{-j}}
\]
for positive constants $ \kappa_j$, \rev{exact expression of which are given in \eqref{kappakj}}. 

\end{lemma}

In the next lemma we give the limiting covariance for the processes $\int_0^t \tilde f_n(s)\der s$.
\begin{lemma}\label{covNoiselemma}
Under the assumptions of Theorem~\ref{white noise thm}, for any $t_1,t_2>0$, 
\begin{align*}
    \lim_{n\to\infty} \ex\left[\int_{0}^{t_1} \int_{0}^{t_2} \tilde f_n(s_1)\tilde f_n(s_2) \der s_1 \der s_2\right] =t_1\wedge t_2. 
\end{align*}
where $t_1\wedge t_2$ is the minimum of $t_1$ and $t_2$.

\end{lemma}

Before we present the proofs of Lemmas~\ref{covOUlemma}--\ref{covNoiselemmabar} \rev{in Sections~\ref{slowcovsubsec} and~\ref{modfastcovsubsec}} we  
calculate the first and second moments of $\wh f_n(t)$, which is related to $\bar f_n(t)$ and $\tilde f_n(t)$ through Proposition~\ref{Propsamedist}. The proof of Lemma~\ref{covNoiselemma} will be presented later \rev{in Section~\ref{modfastcovsubsec}} as it requires more information regarding the normalization $M_n$ appearing in the definition of $\tilde f_n(t)$.

\subsection{Regime independent expressions for the moments}

Here we derive expressions for the first and the second moments of $\wh f_n(t)$ which will serve as a starting point for the proofs of Lemmas~\ref{covOUlemma}-\ref{covNoiselemmabar} and introduce important notation. We will not make any assumptions on the asymptotic relationships between $n$, $r$, and $\Delta$ yet, except that $r \to  0$. Therefore the expressions found for the moments will be valid for all regimes.

Note that, due to stationarity in time (Remark~\ref{stationspacexihat}), and Mecke's formula,
\begin{align*}
 \ex[\wh f_{n}(t)] = \ex[\wh f_{n}(0)] 
 = \frac{n^{k}}{k!}\alpha(r),    
\end{align*}
 where we use $\alpha(r)$ as a shorthand notation for
\[
    \alpha(r)\coloneqq \ex[\wh \xi_r(\bs X)],
\]
  and $\bs X\coloneqq (X_1,\ldots,X_k)$ is a $k$-tuple of iid points uniformly distributed in the cube $Q = [-1/2,1/2)^d\subset\Rd$. Furthermore, we have
\begin{align*}
    \ex[\wh f_{n}(t)^2] &= \ex[\wh f_{n}(0)^2]=  \ex\Bigg[ \sum_{j=0}^{k}\sum_{\substack{\mchX_1, \wh{\mcX}_2  \subseteq \wh{\eta}_{n,T}\\ |\mchX_1 \cap \wh\mcX_2| = j}} \wh\xi_r(\wh\mcX_1(0)) \wh\xi_r(\mchX_2 (0) ) \tau_0(\mchX_1)\tau_0(\mchX_2)\Bigg].
\end{align*}
Using Mecke's formula, and the independent distributions of the marks, we obtain
\[
    \ex[\wh f_{n}(t)^2]= \sum_{j=0}^{k} \frac{[n(1+T)]^{2k-j}}{j!((k-j)!)^2} \alpha_j(r)\frac{1}{(1+T)^{2k-j}}, 
\]
 where
\[
    \alpha_j(r)\coloneqq \ex\left[\wh\xi_r(\bs X)\wh\xi_r (\bs X' )\right],
\]
and where $\bs X\coloneqq (X_1,\ldots,X_k)$  and $\bs X'\coloneqq (X_{1},\ldots, X_j,X_{k+1},\ldots, X_{2k-j})$ are both $k$-tuples of uniform iid points 
in $Q$, sharing $j$ points in common.
Therefore, we obtain
\begin{align}
\var[\wh f_{n}(t)] =  \ex[\wh f_{n}(t)^2] - \frac{n^{2k}}{(k!)^2}[\alpha(r)]^2= \sum\limits_{j=1}^{k}\frac{n^{2k-j}}{j!((k-j)!)^2}    \alpha_j(r), \label{varfntasymp}
\end{align}
for all $t>0$. Furthermore,
\[
    \ex[\wh f_{n}(0)\wh f_{n}(\Delta)] 
    =  \ex\Bigg[\sum_{\mchX_1, \wh{\mcX}_2  \subseteq \wh{\eta}_{n,T}} \wh\xi_r(\wh\mcX_1(0)) \wh\xi_r(\mchX_2 (\Delta) ) \tau_0(\mchX_1)\tau_\Delta(\mchX_2)\Bigg].
\]
Counting through the intersection $|\wh\mcX_1 \cap \wh\mcX_2| = j$ as in the calculation of the variance, and using Mecke's formula, along with the independence of the marks,
\begin{align}
    \ex[\wh f_{n}(0)\wh f_{n}(\Delta)] &=\ex\Bigg[\sum^k_{j=0}  \sum_{\substack{\mchX_1, \wh{\mcX}_2  \subseteq \wh{\eta}_{n,T}\\ |\mchX_1 \cap \wh\mcX_2| = j}} \wh\xi_r(\wh\mcX_1(0)) \wh\xi_r(\mchX_2 (\Delta) ) \tau_0(\mchX_1)\tau_\Delta(\mchX_2) \Bigg]\nonumber\\
    &=  \sum_{j=0}^{k}\frac{[n(1+ T)]^{2k-j}}{j!\big((k-j)!\big)^2}  p_0^{k-j} p_\Delta^{k-j} p_{0,\Delta}^j \ex\left[\wh\xi_r(\bs X) \wh\xi_r\left(\bs X' + \bs{Z}_{\sigma^2\Delta}\right)\right],\label{sumj0k}
\end{align}
where  $\bs{Z}_{\sigma^2\Delta}$ is made up of $k$ independent, 
$d$-dimensional,  Gaussian vectors, representing the displacements of the points in $\bs X'$ between times $0$ and $\Delta$. Thus each vector has zero mean and independent entries with variance $\sigma^2\Delta$. \rev{As a side note, we remark that \eqref{sumj0k} could have been obtained as a consequence of Lemma~3.5 of \cite{reitzner2013}. However, direct computation through Mecke's formula is straightforward and useful for introducing notation that we will need later.} Continuing from \eqref{sumj0k},
\begin{equation*}
\begin{split}
    &p_0 \coloneqq \pr[B_x=0] = \frac{1}{1+ T},   \\
    &p_\Delta \coloneqq \pr[B_x\le \Delta < B_x+L_x] = \frac{1}{1+ T}\exp(-\Delta) + \frac{ T}{1+ T} \int_0^\Delta \frac{1}{T} \exp(-(\Delta-b))db = \frac{1}{1+ T}, \\
    &p_{0,\Delta} \coloneqq \pr[B_x=0, \Delta < L_x] = \frac{\exp(-\Delta)}{1+ T}.
\end{split}
\end{equation*}
Inserting above expressions to \eqref{sumj0k}, we get
\begin{align*}
    \ex[\wh f_{n}(0)\wh f_{n}(\Delta)]
    &=  \sum_{j=0}^{k}\frac{n^{2k-j} \exp(-j\Delta)}{j!\big((k-j)!\big)^2}  \ex\left[\wh \xi_r(\bs X) \wh\xi_r\left(\bs X' + \bs Z_{\sigma^2\Delta}\right)\right].
\end{align*}
Note that, for $j=0$, 
\[\ex\left[\wh \xi_r(\bs X) \wh\xi_r\left(\bs X' + \bs Z_{\sigma^2\Delta}\right)\right] = \ex\left[\wh \xi_r(\bs X)\right] \ex\left[\wh\xi_r\left(\bs X' + \bs Z_{\sigma^2\Delta}\right)\right].\]
Furthermore, 
\begin{align*}
    \ex\left[\wh \xi_r(\bs X)\right] =\ex\left[\wh\xi_r\left(\bs X' + \bs Z_{\sigma^2\Delta}\right)\right] = \alpha(r),
\end{align*}
due to the spatial homogeneity (Remark~\ref{stationspacexihat}).  
Thus, the $j=0$ term in the covariance cancels out to give
\begin{align*}
    \cov[\wh f_n(0), \wh f_n(\Delta)] &= \ex[\wh f_{n}(0)\wh f_{n}(\Delta)] - (\ex[\wh f_{n}(0)])^2 \\
    & = \sum_{j=1}^{k}\frac{n^{2k-j} \exp(-j\Delta)}{j!\big((k-j)!\big)^2}   \ex\left[\wh \xi_r(\bs X) \wh\xi_r\left(\bs X' + \bs Z_{\sigma^2\Delta}\right)\right].
\end{align*}
Note that due to Proposition~\ref{Propsamedist},
\begin{align*}
  \ex[\bar f_n(0) \bar f_n(\Delta)] = \frac{\cov[\wh f_n(0), \wh f_n(\Delta)]}{\var[\wh f_{n}(0)]}
    & = \frac {\sum\limits_{j=1}^{k}\frac{ \exp(-j\Delta)}{j!((k-j)!)^2}   n^{-j}\theta_j(r)} 
    {\sum\limits_{j=1}^{k}\frac{1}{j!((k-j)!)^2}   n^{-j} \alpha_j(r)},
\end{align*}
where
\[\theta_j(r)\coloneqq \ex\left[\wh \xi_r(\bs X) \wh\xi_r\left(\bs X' + \bs Z_{\sigma^2\Delta}\right)\right].\]
\rev{Note that, although suppressed in the notation, $\theta_j(r)$ depends on $\Delta$, $\sigma$, and naturally, $\theta_j(r)=\alpha_j(r)$ when $\Delta=0$.}
Clearly, the  $n\to\infty$ limit of the covariance depends on the asymptotic behavior of $\alpha_j(r)$ and $\theta_j(r)$. First, consider $\alpha_j(r)$. Note that
\[
 \alpha_j(r) = 
  \int_{Q^{2k-j}}  \wh \xi_r( x_1, \ldots,  x_k) 
   \wh \xi_r( x_1, \ldots, x_j,  x_{k+1},\ldots,  x_{2k-j}) \prod_{i=1}^{2k-j} \der x_i. 
\]
Making the change of variables $x_i \to x_1+y_i$ for $2\leq i \leq 2k-j$, and using the translation invariance of $\wh\xi_r$, Remark~\ref{transinvxihat}, we get
\[
    \alpha_j(r) = \int_{Q} \der x_1
  \int_{(Q-x_1)^{2k-j-1}}  \wh\xi_r(0, \bs y) 
   \wh\xi_r(0,\bs y') \prod_{i=2}^{2k-j} \der y_i, 
\]
where $\bs y\coloneqq (y_2,\ldots ,y_k)$ and $\bs y'\coloneqq (y_2,\ldots,y_j,y_{k+1},\ldots,y_{2k-j})$. 
Note that $\wh\xi_r(0, \bs y)$ is a periodic functional of $\bs y$ with period 1 due to its definition \eqref{whxidef}. Thus, we obtain
\[
    \alpha_j(r) = 
  \int_{Q^{2k-j-1}}  \wh\xi_r(0, \bs y) 
   \wh\xi_r(0,\bs y') \prod_{i=2}^{2k-j} \der y_i.
\]
Due to Remark~\ref{localxihat}, for $r$ small enough, namely $r<\frac{1}{2\delta}$,
\[
    \alpha_j(r) = \int_{ B_{\delta r}^{2k-j-1}}  \wh\xi_r(0, \bs y) 
   \wh \xi_r(0, \bs y') \prod_{i=2}^{2k-j} \der y_i,
\]
where we use $B_{\delta r}$ as a shorthand notation for $B_{\delta r}(0)$. 
 Taking the change of variables $\bs y\to r\bs y$ and $\bs y'\to r \bs y'$,
 \[
    \alpha_j(r) = r^{d(2k-j-1)}\int_{ B_{\delta}^{2k-j-1}}  \wh\xi_r(0,  r \bs y) 
   \wh \xi_r(0,  r \bs y') \prod_{i=2}^{2k-j} \der y_i.
\]
 Using Remark~\ref{transinvxihat}, that is the scale invariance of $\wh\xi_r$, for $\bs y, \bs y'\subset B_{\delta}\subset Q$, we conclude that
\begin{align} \label{alphaasymp}
    \alpha_j(r) = r^{d(2k-j-1)} \kappa_{j},
    \end{align}
where
\begin{align}
     \kappa_{j} \coloneqq \int_{B_{\delta}^{2k-j-1}}  \wh\xi_1(0, \bs y) 
   \wh\xi_1(0, \bs y') \prod_{i=2}^{2k-j} \der y_i >0, \label{kappakj}
\end{align} 
which is due to Remark~\ref{feasxihat}. 

Let us now focus on the asymptotics of $\theta_j(r)$.
Using the spatial homogeneity (Remark~\ref{stationspacexihat}),
\begin{align}
 \theta_j(r) = 
  \int_{(\Rd)^j} \der \bs x' \int_{Q^{2k-j}} \hspace{-1em}\der \bs x \der \bs x'' \hspace{1em} \wh \xi_r( \bs x)  \wh \xi_r( \bs x', \bs x'')  \prod_{i=1}^{j} \varphi(x'_i- x_i)  ,
\label{exxi}
\end{align}
where $\bs x\coloneqq (x_1,\ldots ,x_k)$, $\bs x'\coloneqq (x'_1,\ldots ,x'_j)$, $\bs x''\coloneqq (x''_1,\ldots,x''_{k-j})$, and $\varphi:\Rd\to\mathbb{R}^+$ is the symmetric Gaussian density,
\[
\varphi( z)\coloneqq  \frac{1}{(2\pi\Delta)^{d/2}\sigma^d} \exp\Big(-\frac{\|z\|^2}{2\sigma^2\Delta}\Big). 
\]
 Now, we make the following change of variables in \eqref{exxi},
\begin{alignat*}{5}
    x_1&\to u' &&\text{and}\quad
    x'_{1}\to u'+ u, &&\\
    x_i &\to u'+z_{i-1}  \quad&&\textrm{and} \quad x'_i \to u'+u+z'_{i-1}  &&\textrm{for} \quad 2\leq i\leq j,\\
    x_\ell &\to u'+ y_{\ell-j} &&\textrm{and}\quad x''_{\ell-j} \to u'+ u + y'_{\ell-j}  \quad&&\textrm{for} \quad j+1\leq \ell\leq k,
\end{alignat*}
  and, using the translation invariance of $\wh \xi_r$ we get,
\begin{align}\label{exbigxi}
\begin{split}
 \theta_j(r) = 
  \int_{Q} \der u'\int_{\Rd -u'} \hspace{-1em}\der u   &\int_{(\Rd-u-u')^{j-1}} \hspace{-1em}\der \bs z' \int_{(Q-u-u')^{k-j}} \hspace{-1em} \der \bs y' \int_{(Q-u')^{k-1}} \hspace{-1em}\der \bs y \der \bs z \\&\times\wh\xi_r(0,  \bs z, \bs y)
  \wh\xi_r(0, \bs z', \bs y') 
 \varphi(u) \prod_{i=1}^{j-1} \varphi(z'_i-z_i+u )     ,  
 \end{split}
\end{align}
where 
\begin{alignat*}{3}
    \bs y&\coloneqq (y_1,\ldots, y_{k-j}),\quad&&
    \bs z\coloneqq (z_1,\ldots, z_{j-1}),\\
    \bs y'&\coloneqq (y'_{1},\ldots, y'_{k-j}),&&
    \bs z'\coloneqq (z'_1,\ldots, z'_{j-1}).
\end{alignat*}
Due to the periodicity of $\wh \xi_r$, and using the shorthand notation
\begin{align*}
    \bs \varphi(\bs z,u) &\coloneqq \varphi(u) \prod_{i=1}^{j-1} \varphi(z_i+u), 
\end{align*}
we rewrite \eqref{exbigxi} as follows:
\begin{align}\label{exlast}
 \theta_j(r) = 
   &\int_{(\Rd)^{j}} \hspace{-1em} \der u \der \bs z'\int_{Q^{2k-j-1}}\hspace{-1em}\der \bs y \der \bs z \der \bs y' \hspace{0.5em} \wh\xi_r(0,  \bs z, \bs y)
  \wh\xi_r(0, \bs z', \bs y') 
  \bs \varphi(\bs z'-\bs z,u).       
\end{align}

\rev{We conclude by putting together what we have achieved in this section, i.e., a simplified expression for the covariance,
\begin{align}\label{sec42conc}
    \ex[\bar f_n(0) \bar f_n(\Delta)] 
     = \frac {\sum\limits_{j=1}^{k}\frac{ \exp(-j\Delta)}{j!((k-j)!)^2}   n^{-j}\theta_j(r)} 
    {\sum\limits_{j=1}^{k} \frac{ \kappa_{j}}{j!((k-j)!)^2}  n^{-j} r^{d(2k-j-1)}},
\end{align}}
We will derive explicit formulae for $\theta_j(r)$ in the different regimes in the proofs that follow. 

\subsection{Covariance for the slow regime}\label{slowcovsubsec} We prove Lemma~\ref{covOUlemma} here using the notation and identities established in the previous section.  
\efe{The main idea behind the proof will be to separate the domain of the first integral in \eqref{exlast} into a small ball of fixed radius and its complement. It will be possible to exactly calculate the integral inside the ball via a change of variables and exploiting the scale invariance of $\xi_r$, and, using a Gaussian tail bound, we  will show that the complementary integral  is asymptotically negligible.}  

\begin{proof}[Proof of Lemma~\ref{covOUlemma}]

When $\sigma\sqrt{\Delta}/r \to 0$, we make the following change of variables in \eqref{exlast},
\[
    u \to \sigma\sqrt{\Delta} u 
\]
followed by
\[
    z'_i\to \sigma \sqrt{\Delta} v_i - \sigma\sqrt{\Delta}u + z_i \quad\text{for}\quad 1\leq i\leq j-1,
\]
to obtain
\begin{align}\label{star12}
 \begin{split}
 \theta_j(r) = \frac{1
 }{(2\pi)^{dj/2}}
  \int_{(\Rd)^{j}}\hspace{-1em}\der u \der \bs v    \int_{Q^{2k-j-1}}\hspace{-1em}\der \bs y \der \bs z \der \bs y' \hspace{0.5em}&\wh\xi_r(0,  \bs z,\bs y)
  \wh\xi_r\left(0,  \sigma \sqrt{\Delta} (\bs 
  v -u) +\bs z, \bs y'\right) 
   \\
   &\times\exp\left( - \|\bs v\|^2/2\right) \exp\left( - \|u\|^2/2\right),   
 \end{split}
\end{align}
where $\bs v\coloneqq (v_1,\ldots, v_{j-1})$, and
\begin{align}
    \exp\left( - \|\bs v\|^2/2\right) \coloneqq \exp\left( - \frac{1}{2}\sum_{i=1}^{j-1} \| v_i\|^2\right). 
\end{align}
Due to the locality of $\wh\xi_r$, the integration over $Q$ in the second integral in \eqref{star12} can be restricted to $B_{\delta r}$, and, with the change of variables 
$\bs y\to r\bs y,
\bs y'\to r\bs y', \bs z\to r\bs z$, we can write
\begin{align}\label{star140}
 \begin{split}  
 \theta_j(r)  = &\frac{r^{d(2k-j-1)}
 }{(2\pi)^{dj/2}}
  \int_{(\Rd)^{j}}\hspace{-1em}\der u \der \bs v    \int_{B_\delta^{2k-j-1}}\der \bs y \der \bs z \der \bs y' \hspace{0.5em} \wh\xi_r(0,  r\bs z,r\bs y)\\
  &\times\wh\xi_r\left(0,  \sigma \sqrt{\Delta} (\bs v -u) +r\bs z, r\bs y' \right) 
   \exp\left( - \|\bs v\|^2/2\right) \exp\left( - \|u\|^2/2\right).
 \end{split}
\end{align}
Note that the integrand in \eqref{star140} is  bounded above by
\[
    \|\xi\|_{\infty}^2 \ind\{\bs y,\bs y', \bs z\in B_{\delta}\}  \exp\left( - \|\bs v\|^2/2\right) \exp\left( - \|u\|^2/2\right),
\]
which is integrable over $(\Rd)^{j}\times B_{\delta }^{2k-j-1} $. Using the Dominated Convergence Theorem (DCT), this is enough to conclude that $\theta_j(r)\asymp r^{d(2k-j-1)}$. However, in order to find the exact constant to which $\theta_j(r)/ r^{d(2k-j-1)}$ converges, we need to use the scale invariance of $\wh\xi_r$ \eqref{scaleinvxihat}. As \eqref{scaleinvxihat} is valid only in $Q$, we separate the integral in \eqref{star140} into two parts, as follows: 
\begin{align}\label{bdelta3}
 \begin{split}  
  &\int_{B_{(r/\sigma)^{1/2}}^{j}}  \hspace{-1em}\der u \der \bs v    \int_{B_\delta^{2k-j-1}}\hspace{-1em}\der \bs y \der \bs z \der \bs y' \hspace{0.5em} \wh\xi_r(0,  r\bs z,r\bs y)
  \wh\xi_r\left(0,  \sigma \sqrt{\Delta} (\bs v -u) +r\bs z, r\bs y' \right)e^{-\frac{\|\bs v\|^2 + \|u\|^2}{2}} 
   \\
     +&\int_{(\Rd)^j \setminus B_{(r/\sigma)^{1/2}}^{j}}  \hspace{-2em}\der u \der \bs v    \int_{B_\delta^{2k-j-1}}\hspace{-1em}\der \bs y \der \bs z \der \bs y' \hspace{0.5em} \wh\xi_r(0,  r\bs z,r\bs y)
  \wh\xi_r\left(0,  \sigma \sqrt{\Delta} (\bs v -u) +r\bs z, r\bs y' \right) e^{-\frac{\|\bs v\|^2 + \|u\|^2}{2}}.
 \end{split}
\end{align}
For the first integral in \eqref{bdelta3} we note that for all $\bs v \in B_{(r/\sigma)^{1/2}}^{j-1}$, $u \in B_{(r/\sigma)^{1/2}}$, $\bs z
\in B_{\delta }^{2k-j-1}$, all the points that compose $\bs z, \bs y, \bs y',$ and $\sigma \sqrt{\Delta}/r (\bs v -u) +\bs z$ are inside $Q$.
Therefore, using the scale invariance of $\wh \xi_r$, we obtain that the first term in \eqref{bdelta3} is equal to
\begin{align*}
 \begin{split}  
  \int_{B_{(r/\sigma)^{1/2}}^{j}}  \hspace{-1em}\der u \der \bs v    \int_{B_\delta^{2k-j-1}}\hspace{-1em}\der \bs y \der \bs z \der \bs y' \hspace{0.5em} \wh\xi_1(0,  \bs z,\bs y)
  \wh\xi_1\left(0,   \sigma \sqrt{\Delta}/r (\bs v -u) +\bs z, \bs y'\right) e^{-\frac{\|\bs v\|^2 + \|u\|^2}{2}}. 
 \end{split}
\end{align*}   
Using DCT for the expression above and since $\sigma\sqrt{\Delta}/r \to 0$, we conclude that 
\begin{align}\nonumber
\begin{alignedat}{3}
  &\lim_{n\to\infty}\int_{B_{(r/\sigma)^{1/2}}^{j}}  \hspace{-1em}\der u \der \bs v    \int_{B_\delta^{2k-j-1}}\hspace{-1em}\der \bs y \der \bs z \der \bs y' \hspace{0.5em} &&\mathrlap{\wh\xi_r(0,  r\bs z,r\bs y)
  \wh\xi_r\left(0,  \sigma \sqrt{\Delta} (\bs v -u) +r\bs z, r\bs y' \right) e^{-\frac{\|\bs v\|^2 + \|u\|^2}{2}}}\hphantom{\wh\xi_r(0,  r\bs z,r\bs y)
  \wh\xi_r\left(0,  \sigma \sqrt{\Delta} (\bs v -u) +r\bs z, r\bs y' \right) e^{-\frac{\|\bs v\|^2 + \|u\|^2}{2}}}\\
   &=\lim_{n\to\infty} \int_{(\Rd)^{j}}  \hspace{-1em}\der u \der \bs v    \int_{B_\delta^{2k-j-1}}\hspace{-1em}\der \bs y \der \bs z \der \bs y' \hspace{0.5em}  &&\mathrlap{\wh\xi_1(0,  \bs z,\bs y)
  \lim_{n\to\infty}  \wh\xi_1\left(0,   \sigma \sqrt{\Delta}/r (\bs v -u) +\bs z, \bs y'\right)} \\&&&\times\ind\{\bs v\in B_{(r/\sigma)^{1/2}}^{j}\}
   \ind\{ u\in B_{(r/\sigma)^{1/2}}\} e^{-\frac{\|\bs v\|^2 + \|u\|^2}{2}} \\
   &=\int_{(\Rd)^{j}}  \hspace{-1em}\der u \der \bs v    \int_{B_\delta^{2k-j-1}}\hspace{-1em}\der \bs y \der \bs z \der \bs y' \hspace{0.5em} \mathrlap{  \wh\xi_1(0,  \bs z,\bs y)
  \wh\xi_1\left(0,  \bs z, \bs y'\right)  e^{-\frac{\|\bs v\|^2 + \|u\|^2}{2}}}
  \\ &=  \kappa_j(2\pi)^{dj/2}, 
  \end{alignedat}
\end{align}
with $\kappa_j$ as defined in \eqref{kappakj}. Note that in the last step above we used Remark~\ref{aecontremark}.  
Now, we observe that the second term of \eqref{bdelta3} can be bounded above by   
\begin{align*}
    \int_{(\Rd)^j \setminus B_{(r/\sigma)^{1/2}}^{j}}  \hspace{-1em}\der u \der \bs v    \int_{B_\delta^{2k-j-1}}\hspace{-1em}\der \bs y \der \bs z \der \bs y' \hspace{0.5em}  \|\xi\|_{\infty}^2
   e^{-\frac{\|\bs v\|^2 + \|u\|^2}{2}}   \leq K \sqrt{\frac{\sigma}{r}} e^{-\frac{r}{\sigma d}} \to 0,
\end{align*}
for some constant $K>0$, using the classical Mills inequality on the tail of the Gaussian density \cite{mills}. With these, we conclude that $\theta_j(r) r^{-d(2k-j-1)} \approx 
   \kappa_j$, and we obtain, \rev{from \eqref{sec42conc},} that
\begin{align*} 
    \ex[\bar{f}_n(0)& \bar{f}_n(\Delta)]\\ 
     &\approx \frac {\sum\limits_{j=1}^{k} \frac{ \kappa_{j}}{j!((k-j)!)^2 } \exp(-j\Delta)    n^{-j} r^{d(2k-j-1)}} 
    {\sum\limits_{j=1}^{k} \frac{ \kappa_{j}}{j!((k-j)!)^2 }   n^{-j} r^{d(2k-j-1)}}
   \approx \begin{cases}\exp(-k\Delta) \quad&  nr^d \to 0,\\
    \sum_{j=1}^{k} \lambda_j \exp(-j\Delta) \quad&  nr^d \to \gamma>0,\\
    \exp(-\Delta)\quad&  nr^d \to \infty,
    \end{cases}
\end{align*}
where 
\rev{
\begin{alignat}{1}\label{lambdaexact}
   \lambda_j\coloneqq \frac { \frac{ \kappa_{j}}{j!((k-j)!)^2 }  \gamma^{-j}} 
    {\sum\limits_{\ell=1}^{k} \frac{ \kappa_{\ell}}{j!((k-\ell)!)^2 } \gamma^{-\ell}} = \frac {  \frac{\int_{B_{\delta}^{2k-j-1}}  \wh\xi_1(0, \bs y) 
   \wh\xi_1(0, \bs y') \prod_{i=2}^{2k-j} \der y_i}{j!((k-j)!)^2 }  \gamma^{-j}} 
    {\sum\limits_{\ell=1}^{k}  \frac{\int_{B_{\delta}^{2k-\ell-1}}  \wh\xi_1(0, \bs w) 
   \wh\xi_1(0, \bs w') \prod_{i=2}^{2k-\ell} \der w_i}{\ell!((k-\ell)!)^2 } \gamma^{-\ell}}, 
\end{alignat}
with $\bs w\coloneqq (w_2,\ldots ,w_k)$ and $\bs w'\coloneqq (w_2,\ldots,w_\ell,w_{k+1},\ldots,w_{2k-\ell}).$}
This concludes the proof of Lemma~\ref{covOUlemma}. 
\end{proof}

\subsection{Covariance for the moderate and fast regimes}\label{modfastcovsubsec}
We first give a lemma for the asymptotics of $\theta_j(r)$ that is valid in both regimes, which will  eventually make the proofs of Lemmas~\ref{covModlemma} and \ref{covNoiselemmabar} somewhat more concise. \efe{The basic idea behind the proof of these lemmas will be the same as that of Lemma~\ref{covOUlemma}, but here we will need to take care of the periodicity issues after rescaling, since we work on flat torus.} 

\begin{lemma}\label{intermlem}
    Assume $1\lesssim \sigma\sqrt{\Delta}/r$, and take some $R\to 0$ as $n\to\infty$. Then $\theta_j(r)=\vartheta_j(r)+\vartheta_j'(r)$, where
    \begin{align}\label{star1}
 \begin{split}
 \vartheta_j(r) \approx &\frac{r^{d(2k-1)}}{(\sqrt{2\pi} \sigma)^{dj}\Delta^{dj/2}}
  \int_{B_{\delta}^{2k-2}} \hspace{-0.5em}\der \bs y \der \bs z \der \bs y'  \der \bs z' \hspace{0.5em} \wh\xi_1(0, \bs z, \bs y) \wh\xi_1(0, \bs z', \bs y') \\
  \times&\int_{B_{2\delta +R/r}} \exp\left( -  \frac{r^2}{2\Delta\sigma^2} \left(\|u\|^2 + \sum_{i=1}^{j-1}  \|z'_{i}+u-z_i\|^2\right)\right)  \der u,
 \end{split}
\end{align}
and
\begin{align}\label{star2}
    \vartheta_j'(r) \lesssim r^{d(2k-2)} \sigma^{-dj} \Delta^{-dj/2} \exp\left( -  \frac{jR^2 }{2\Delta\sigma^2}\right).
\end{align}
\end{lemma}
\begin{proof}
    
From \eqref{exlast}, and the locality of $\wh\xi_r$, we obtain
\begin{align*}
 \theta_j(r) = 
   \int_{(\bigcup_{\varrho\in\mathbb{Z}^d} B_{\delta r} (\varrho) )^{j-1}} \hspace{-1em}\der \bs z'\int_{ B_{\delta r}^{2k-j-1}}\hspace{-0.5em}  \der \bs y \der \bs z \der \bs y'\hspace{0.5em} \wh\xi_r(0,  \bs z,\bs y)
  \wh\xi_r(0, \bs z', \bs y') 
  \int_{\Rd} \bs{\varphi}(\bs z'-\bs z,u)  \der u  .
\end{align*}
Due to the periodicity of $\wh\xi_r$,
\begin{align*}
 \theta_j(r) = 
   \int_{ B_{\delta r}^{2k-2}} \hspace{-1em}\der \bs y \der \bs z \der \bs y'  \der \bs z' \hspace{0.5em}\wh\xi_r(0,  \bs z,\bs y)
  \wh\xi_r(0, \bs z', \bs y')
 \int_{\Rd} \sum_{\bs\varrho\in\mathbb{Z}^{d(j-1)}} \bs{\varphi}(\bs z'+\bs \varrho-\bs z,u)  \der u .
\end{align*}
We define the following shorthand notation
\begin{align*}
    \tilde{B}_{R+2\delta r}\coloneqq \bigcup_{\nu\in \mathbb{Z}^d} B_{R+2\delta r}(\nu),
\end{align*}
for some $R$ such that $\lim\limits_{n\to\infty}R= 0$, and  separate the domain of integration of $u$ into two parts,
\begin{align*}
\begin{alignedat}{2}
 \theta_j(r) = &
  \int_{B_{\delta r}^{2k-2}} \hspace{-0.5em}\der \bs y \der \bs z \der \bs y'  \der \bs z' \hspace{0.5em} \wh \xi_r(0,  \bs z, \bs y) 
  \wh\xi_r(0, \bs z', \bs y')\int_{\tilde B_{R+2\delta r}} \sum_{\bs\varrho\in\mathbb{Z}^{d(j-1)}} \bs{\varphi} (\bs z'+\bs \varrho-\bs z,u)  \der u   \\  
 +&  
  \int_{B_{\delta r}^{2k-2}} \hspace{-0.5em}\der \bs y \der \bs z \der \bs y'  \der \bs z' \hspace{0.5em} \wh\xi_r(0, \bs z, \bs y) 
  \wh\xi_r(0, \bs z', \bs y')
\int_{\Rd\setminus \tilde B_{R+2\delta r}}\sum_{\bs\varrho\in\mathbb{Z}^{d(j-1)}} \bs{\varphi} (\bs z'+\bs \varrho-\bs z,u)  \der u    
 \\
 =&
 \int_{B_{\delta r}^{2k-2}}  \hspace{-0.5em}\der \bs y \der \bs z \der \bs y'  \der \bs z' \hspace{0.5em}\wh \xi_r(0,  \bs z, \bs y) 
  \wh\xi_r(0, \bs z', \bs y')\\ 
 &\hspace{10em}\times\int_{B_{R+2\delta r}} \sum_{\bs\varrho\in\mathbb{Z}^{d(j-1)}} \sum_{\nu\in \mathbb{Z}^d}\bs{\varphi} (\bs z'+\bs \varrho-\bs z,u+\nu)  \der u    \\  
 +&  
  \int_{B_{\delta r}^{2k-2}}  \hspace{-0.5em}\der \bs y \der \bs z \der \bs y'  \der \bs z' \hspace{0.5em} \wh\xi_r(0, \bs z, \bs y) 
  \wh\xi_r(0, \bs z', \bs y')\\
 &\hspace{10em}\times\int_{Q\setminus B_{R+2\delta r}}\sum_{\bs\varrho\in\mathbb{Z}^{d(j-1)}} \sum_{\nu\in \mathbb{Z}^d} \bs{\varphi} (\bs z'+\bs \varrho-\bs z,u+\nu)  \der u  .
 \end{alignedat}
\end{align*}
Introducing the shorthand notation,
\begin{align*}
    \tilde \varphi(u)\coloneqq \sum_{\nu\in \mathbb{Z}^d}\varphi(u+\nu),
\end{align*}
we observe that
\begin{align*}
\begin{alignedat}{2}
    &\sum_{\bs\varrho\in\mathbb{Z}^{d(j-1)}} \sum_{\nu\in \mathbb{Z}^d} \bs{\varphi}(\bs z'+\bs \varrho-\bs z,u+\nu) \\ &= \sum_{\nu\in \mathbb{Z}^d}\varphi( u+\nu)  
    \sum_{\bs\varrho\in\mathbb{Z}^{d(j-1)}} \prod_{i=1}^{j-1} \varphi( z'_i+\varrho_i + u + \nu-z_{i}) = \sum_{\nu\in \mathbb{Z}^d}\varphi( u+\nu) \prod_{i=1}^{j-1}  \tilde \varphi( z'_i+u-z_{i}) ,
    \end{alignedat}
\end{align*}
and, therefore,
\begin{align}
\begin{alignedat}{2}\label{xikr}
 \theta_j(r)  
 =&\int_{B_{\delta r}^{2k-2}} \hspace{-0.5em}\der \bs y \der \bs z \der \bs y'  \der \bs z' \hspace{0.5em} \wh \xi_r(0,  \bs z, \bs y) 
  \wh\xi_r(0, \bs z', \bs y')\int_{B_{R+2\delta r}} \tilde \varphi( u) \prod_{i=1}^{j-1}   \tilde \varphi( z'_i+u-z_{i})  \der u   \\  
 +&  
  \int_{B_{\delta r}^{2k-2}} \hspace{-0.5em}\der \bs y \der \bs z \der \bs y'  \der \bs z' \hspace{0.5em} \wh\xi_r(0, \bs z, \bs y) 
  \wh\xi_r(0, \bs z', \bs y')\int_{Q\setminus B_{R+2\delta r}} \tilde \varphi( u) \prod_{i=1}^{j-1} \tilde \varphi( z'_i+u-z_{i}) \der u.
 \end{alignedat}
\end{align}

 Next, we consider the asymptotic behavior of ${\tilde\varphi}$ in each integral in \eqref{xikr}. Note that, for any $u\in \Rd$,
\begin{align*}
    &\tilde \varphi( u)=\sum_{\nu\in\mathbb{Z}^{d}} \frac{1}{(\sqrt{2\pi} \sigma)^{d}\Delta^{d/2}} \exp\left( -  \frac{\|u+\nu\|^2}{2\Delta\sigma^2}   \right)\\
    &= \frac{1}{(\sqrt{2\pi} \sigma)^{d}\Delta^{d/2}}\left[ \exp\left( -  \frac{\|u\|^2}{2\Delta\sigma^2}   \right)+
    \sum_{\ell=1}^\infty \sum_{\nu\in\zd}\ind\left\{ \sum_{i=1}^d |\nu_i|= \ell\right\} \exp\left( -  \frac{\|u+\nu\|^2}{2\Delta\sigma^2} \right) \right].
\end{align*}
 For all $\nu\in \zd$ such that $\sum_{i=1}^d |\nu_i|=\ell$, $\|u\|\leq \frac{1}{4d}$, and $\ell\geq 1$
\begin{align*}
    \|u+\nu\|^2\geq \|\nu\|^2 + \|u\|^2 -2\|u\|\|\nu\| &\geq \frac{\ell^2}{d} + \|u\|^2 - 2\|u\| \ell\\
    &\geq \frac{\ell^2}{d} + \|u\|^2 - \frac{\ell}{2d} \geq \frac{\ell}{2d} + \|u\|^2.
\end{align*}
Therefore,
\begin{align*}
    \sum_{\ell=1}^\infty &\sum_{\nu\in\zd}\ind\left\{ \sum_{i=1}^d |\nu_i|= \ell\right\}\nonumber \exp\left( -  \frac{\|u+\nu\|^2}{2\Delta\sigma^2} \right)\\ \leq& \sum_{\ell=1}^{\infty} (2\ell+1)^{d} \exp\left( -  \frac{\ell}{4 d \Delta\sigma^2} - \frac{\|u\|^2}{2\Delta\sigma^2} \right)
    \leq 3^d \exp\left( - \frac{\|u\|^2}{2\Delta\sigma^2} \right) \sum_{\ell=1}^{\infty} \ell^d \exp\left( -  \frac{\ell}{4 d \Delta\sigma^2}\right).
\end{align*}
Note that, 
\begin{align*} 
    \sum_{\ell=1}^{\infty} \ell^d \exp\left( -  \frac{\ell}{4 d \Delta\sigma^2}\right) 
    =  \sum_{\ell=1}^{d^2} \ell^d \exp\left( -  \frac{\ell}{4 d \Delta\sigma^2}\right) + \sum_{\ell=d^2+1}^{\infty} \ell^{d} \exp\left( -  \frac{\ell}{4 d \Delta\sigma^2}\right). 
\end{align*}
The first sum on the right hand side above converges to $0$ as $\sigma$ goes to $0$, and since
\begin{align*}
    d \log \ell \leq \sqrt{\ell} \log \ell \leq \ell 
\end{align*}
for all $\ell> d^2$, the second sum is bounded above by
\begin{align*}
    \sum_{\ell=1}^{\infty} \exp(\ell) \exp\Big( -  \frac{\ell}{4 d \Delta\sigma^2}\Big),
\end{align*}
which is a geometric sum that also converges to $0$. Therefore,
\begin{align}
   \tilde \varphi( u) \approx \frac{1}{(\sqrt{2\pi} \sigma)^{d}\Delta^{d/2}} \exp\left( -  \frac{\|u\|^2}{2\Delta\sigma^2},   \right)\label{psigma2}
\end{align}
for all $\|u\|\leq \frac{1}{4d}$. Note that in the domain of integration for the first term in \eqref{xikr} 
\begin{align*}
    \|u\|&\leq R+2\delta r =  o(1),\\
    \|z'_i+u-z_{i}\|&\leq R+4\delta r = o(1).
\end{align*}
Therefore, we can apply \eqref{psigma2} to the first term in \eqref{xikr}, which we write  as
\begin{align*}
\begin{split}
 \vartheta_j(r) \approx 
  \int_{B_{\delta r}^{2k-2}} \hspace{-0.5em}\der \bs y \der \bs z \der \bs y'  \der \bs z' \hspace{0.5em} &\wh\xi_r(0, \bs z, \bs y) \wh\xi_r(0, \bs z', \bs y') \frac{1}{(\sqrt{2\pi} \sigma)^{dj}\Delta^{dj/2}}\\
  &\times\int_{B_{R + 2\delta r}}  \exp\left( -  \frac{\sum_{i=1}^{j-1}  \|z'_{i}+u-z_i\|^2}{2\Delta\sigma^2} - \frac{\|u\|^2}{2\Delta\sigma^2} \right)  \der u .
\end{split}
\end{align*}
We make the change of variables $
\bs y\to r\bs y,
\bs y'\to r\bs y', \bs z\to r\bs z,
\bs z'\to r\bs z'$ and $u\to ru$,
and use the scale invariance property of $\wh \xi_r$ to obtain \eqref{star1}.

Now consider the second term in \eqref{xikr}, which will be denoted by $\vartheta_j'(r)$. The following upper bound holds for any $z',z\in B_{\delta r}$ and $u\in Q\setminus B_{R+2\delta r}$:
\begin{align*}
    \max\left(\tilde \varphi(u), \tilde \varphi(z'+u - z)\right) \leq \tilde \varphi( w)
\end{align*}
where $w\in\Rd$ is an arbitrary point with $\|w\|=R$.
Therefore,
\begin{align*}
    \vartheta_j'(r)\leq [\tilde \varphi(w)]^j\int_{B_{\delta r}^{2k-2}}  \wh\xi_r(0, \bs z, \bs y) 
  \wh\xi_r(0, \bs z', \bs y')
  \der \bs y \der \bs z \der \bs y'  \der \bs z'.
\end{align*}
Using \eqref{psigma2}, we can write
\begin{align*}
    \vartheta_j'(r)\lesssim\frac{1}{(\sqrt{2\pi} \sigma)^{dj}\Delta^{dj/2}} \exp\left( -  \frac{jR^2 }{2\Delta\sigma^2}\right)
 \left(\int_{B_{\delta r}^{k-1}}  \wh\xi_r(0,  \bs y) 
   \der \bs y \right)^2.
\end{align*}
With the change of variables, $\bs y\to r\bs y$, we obtain \eqref{star2}.
\end{proof}

Next, we make specific assumptions on $R$ to prove Lemmas~\ref{covModlemma} and \ref{covNoiselemmabar} separately. 
\begin{proof}[Proof of Lemma~\ref{covModlemma}]
In the \emph{moderate regime} where $\sigma\sqrt{\Delta}/r \to \sqrt{\beta}\in(0,\infty)$, we take 
\begin{align} \label{Rregime1}
    R\coloneqq \sigma\sqrt{\Delta r^{-\epsilon}},
\end{align}
for a fixed $0<\epsilon<1$.
Note that $R\to 0$ as required by Lemma~\ref{intermlem}, due to the definition of $R$ and the regime we work in.
Furthermore,
\begin{align}
  \|u\|^2 + \sum_{i=1}^{j-1}  \|z'_{i}+u-z_i\|^2 &=  \|u\|^2 + \sum_{i=1}^{j-1}  \left[\|u\|^2 - 2u^\top (z'_{i}-z_i) + \|z'_{i}-z_i\|^2 \right]\nn\\
    &=j\left(  \|u\|^2 - 2u^\top \sum_{i=1}^{j-1} \frac{z'_{i}-z_i}{j} \right)+ \sum_{i=1}^{j-1} \left\|z'_{i}-z_i\right\|^2 \nn\\
    &=j \left\|u- \sum_{i=1}^{j-1} \frac{z'_i-z_i}{j}\right\|^2 +  \sum_{i=1}^{j-1} \left\|z'_{i}-z_i\right\|^2 -j \left\|\sum_{i=1}^{j-1} \frac{z'_i-z_i}{j}\right\|^2.\label{uzexpansion} 
\end{align}
Therefore, using $R/r \to \infty$,
\begin{align*}
\begin{split}
 &\lim_{n\to\infty} 
  \int_{B_{2\delta +R/r}}  \exp\left( -  \frac{r^2}{2\Delta\sigma^2} \left(\|u\|^2 + \sum_{i=1}^{j-1}  \|z'_{i}+u-z_i\|^2\right)\right)  \der u   \\
  &=\lim_{n\to\infty} 
  \int_{B_{2\delta +R/r}} \exp\left( -  \frac{jr^2}{2\Delta\sigma^2}\left\|u- \sum_{i=1}^{j-1} \frac{z'_i-z_i}{j}\right\|^2   \right)  \der u  \\
  &\hphantom{=}\times\lim_{n\to\infty} \exp\left( -\frac{r^2}{2\Delta\sigma^2}\left(\sum_{i=1}^{j-1} \left\|z'_{i}-z_i\right\|^2 -j \left\|\sum_{i=1}^{j-1} \frac{z'_i-z_i}{j}\right\|^2\right)  \right)  \\
  &=\left(\frac{2\pi\beta}{j}\right)^{d/2} \exp\left( -\frac{1}{2\beta}\left(\sum_{i=1}^{j-1} \left\|z'_{i}-z_i\right\|^2 -j \left\|\sum_{i=1}^{j-1} \frac{z'_i-z_i}{j}\right\|^2\right)  \right).
  \end{split}
\end{align*}
Thus, from \eqref{star1},
\begin{align}\label{varthetajapprox}
    \vartheta_j(r) \cdot r^{-d(2k-j-1)} \approx \tilde \zeta_j(\beta),
\end{align}
where 
\begin{align*}
     \begin{split}
 \tilde \zeta_j(\beta)\coloneqq & 
 (2\pi \beta)^{-d(j-1)/2} j^{-d/2}
  \int_{B_\delta^{2k-2}}  \wh\xi_1(0, \bs z, \bs y) \wh\xi_1(0, \bs z', \bs y') \\
  &\times\exp\left( -\frac{1}{2\beta}\left(\sum_{i=1}^{j-1} \left\|z'_{i}-z_i\right\|^2 -j \left\|\sum_{i=1}^{j-1} \frac{z'_i-z_i}{j}\right\|^2\right)  \right) \der \bs y \der \bs z \der \bs y'  \der \bs z'.
 \end{split}
\end{align*}
Note that for $j=1$, $\tilde \zeta_j(\beta) = \kappa_j$ (see \eqref{kappakj}). For $j\geq 2$, we can write $\tilde \zeta_j(\beta)$ more concisely as
\begin{align}
\begin{split}\label{zetajtilde}
 \tilde \zeta_j(\beta)= &
  \int_{B_\delta^{2k-2}}  \wh\xi_1(0, \bs z, \bs y) \wh\xi_1(0, \bs z', \bs y')\\
  \times&\prod_{i=1}^d \frac{1}{(2\pi)^{(j-1)/2}\sqrt{\beta^{j-1}j} } \exp\left(-\frac{1}{2} \sum_{\ell, l=1}^{j-1} (z_{\ell,i} - z'_{\ell,i}) \mathrm{M}^{j,\beta}_{\ell,l} (z_{l,i} - z'_{l,i}) \right)  \der \bs y \der \bs z \der \bs y'  \der \bs z',
 \end{split}
\end{align}
where $z_{\ell,i}\in \mathbb{R}$ denotes the $i$-th entry of $ z_\ell\in \Rd$, and 
\begin{align*}
    \mathrm{M}^{j,\beta}_{\ell,l} \coloneqq 
    \begin{cases}
    \frac{j-1}{j\beta} \quad &\textrm{if} \quad \ell=l\\
    -\frac{1}{j\beta} \quad &\textrm{if} \quad \ell\neq l.
    \end{cases}
\end{align*}
The (invertible) matrix $\mathrm{M}^{j,\beta}$ can also be written as
\begin{align}\label{matrixMdef}
    \mathrm{M}^{j,\beta} = \frac{1}{j\beta} (j\mathrm{I}_{j-1} - \mathrm{J}_{j-1} )  ,
\end{align}
where $\mathrm{I}_{j-1}$ denotes the $(j-1)\times(j-1)$ identity matrix, and $\mathrm{J}_{j-1}$ denotes the matrix with the same dimensions but with all entries  $1$. Note that \rev{from the expression defining $\mathrm{M}^{j,\beta}$, adding all the columns to the first column, and then subtracting the first row from the other rows, we obtain an upper-triangular matrix with diagonal entries $\frac{1}{j\beta} , \frac{j}{j\beta},\ldots, \frac{j}{j\beta}$.} Therefore
$    \det|\mathrm{M}^{j,\beta}| = \frac{1}{\beta^{j-1}j}$, and \eqref{zetajtilde} can be rewritten as 
 \begin{align}\label{zetajprob}
    \tilde \zeta_j(\beta) = \ex\left[\wh\xi_1(0, \bs z, \bs y) \wh\xi_1(0, (\bs z + \bs w_{\beta}) , \bs y')  \right],
\end{align}
where $\bs w_\beta$ is a set of $(j-1)$ jointly Gaussian distributed vectors in $\Rd$ with the inverse covariance matrix $\mathrm M^{j,\beta}$ as  given in \eqref{matrixMdef}, and $\bs z, \bs y, \bs y'$ are iid points in $Q\subset \Rd$.

 Due to the feasibility of $\wh\xi_1$, $ \tilde \zeta_j(\beta)$ is strictly positive for all $\beta>0$, which gives us
$\vartheta_j(r) \asymp  r^{d(2k-j-1)}$.
With the  definition of $R$ given in \eqref{Rregime1}, it follows that
\begin{align*}
    \exp\left( -  \frac{jR^2 }{2\Delta\sigma^2}\right) = \exp\left( -  \frac{j}{2} r^{-\epsilon}\right)  \ll r^d. 
\end{align*}
Therefore $\vartheta_j'(r) \ll r^{d(2k-1)}\sigma^{-dj}\Delta^{-dj/2}$,  so that $\vartheta_j'(r)  \ll \vartheta_j(r)$, which, due to \eqref{varthetajapprox} \rev{and \eqref{sec42conc}}, gives
\begin{align*}
    \ex[\bar{f}_n(0) \bar{f}_n(\Delta)]
    &\approx \frac {\sum\limits_{j=1}^{k} \exp(-j\Delta) \frac{\tilde \zeta_j( \beta)}{j!((k-j)!)^2}   n^{-j} r^{d(2k-j-1)}} 
    {\sum\limits_{j=1}^{k} \frac{ \kappa_{j}}{j!((k-j)!)^2 }   n^{-j} r^{d(2k-j-1)}} \approx \sum\limits_{j=1}^{k} \exp(-j\Delta)  \zeta_j( \beta) \lambda_j,
\end{align*}
where
\begin{align}
\label{definezeta}
    \zeta_j(\beta)\coloneqq  \frac{\tilde \zeta_j(\beta)}{\kappa_j}.
\end{align}
Recall that $\kappa_j$ and $\lambda_j$ were defined at the end of the Proof of Lemma~\ref{covOUlemma}. 
Since $\tilde \zeta_1(\beta) = \kappa_1$ for all $\beta>0$, $\zeta_1(\beta)=1$. 
With this, we conclude that if $ \sigma\sqrt{\Delta}/r\to  \sqrt{\beta}\in(0,\infty)$,
\begin{align*}
    \ex[\bar{f}_n(0) \bar{f}_n(\Delta)]
    & \approx \lambda_1e^{-\Delta} + \sum\limits_{j=2}^{k} \exp(-j\Delta)  \zeta_j( \beta) \lambda_j
\end{align*}
for strictly positive functions $\zeta_j(\beta)$. \rev{Note that $\lim_{\beta\to 0^+} \zeta_j(\beta)= 1$, and, naturally, our covariance expressions for the moderate and slow regime coincide when $\beta\to 0^+$.}
\end{proof}

\begin{proof}[Proof of Lemma~\ref{covNoiselemmabar}]
In the fast regime, where
  $r\ll  \sigma\sqrt{\Delta} \ll 1$,
we will again apply the expansion \eqref{uzexpansion} to \eqref{star1}. 
Note that
\begin{align}\label{expto1}
    \exp\left( -\frac{r^2}{2\Delta\sigma^2}\left(\sum_{i=1}^{j-1} \left\|z'_{i}-z_i\right\|^2 -j \left\|\sum_{i=1}^{j-1} \frac{z'_i-z_i}{j}\right\|^2\right)  \right) \to 1,
\end{align}
and 
\begin{align}
     &\left(\frac{2\pi\Delta}{j}\right)^{-d/2} \left(\frac{r}{\sigma}\right)^d
  \int_{B_{2\delta +R/r}} \exp\left( -  \frac{jr^2}{2\Delta\sigma^2}\left\|u- \sum_{i=1}^{j-1} \frac{z'_i-z_i}{j}\right\|^2   \right)  \der u \nn \\
   &=\left(\frac{2\pi\Delta}{j}\right)^{-d/2} \left(\frac{r}{\sigma}\right)^d
  \int_{B_{2\delta  +R/r}(m_{\bs z})} \exp\left( -  \frac{jr^2}{2\Delta\sigma^2}\left\|v\right\|^2   \right)  \der v   \nn\\
    &=\left(2\pi\right)^{-d/2} 
  \int_{\Rd} \exp\left( -  \frac{\|v\|^2}{2}   \right) \ind\left\{\left\|v-\sqrt{\frac{j}{\Delta}}\frac{r}{\sigma} m_{\bs z}\right\|\leq  \left(2\delta +\frac{R}{r}\right) \sqrt{\frac{j}{\Delta}}\frac{r}{\sigma}  \right\}  \der v,\label{eqlemm43}
\end{align}
where
$m_{\bs z}\coloneqq \sum_{i=1}^{j-1} \frac{z'_i-z_i}{j} \in \Rd.$
Under the choice of
$
    R\coloneqq \sigma\sqrt{\Delta} \log\frac{1}{\sigma\sqrt{\Delta}}
$,
 note also that
$
\frac{R}{r} \sqrt{\frac{j}{\Delta}}\frac{r}{\sigma}\to\infty
$.
Therefore \eqref{eqlemm43} converges to $1$, which together with \eqref{expto1}, establishes the asymptotic behavior of the second integral in \eqref{star1}, giving
\begin{align*}
 \frac{\vartheta_j(r)} {r^{d(2k-2)}(\sigma\sqrt{\Delta})^{-d(j-1)}}&\to  (2\pi)^{d(1-j)/2} j^{-d/2} 
    \int_{B_{\delta}^{2k-2}}  \wh\xi_1(0, \bs z, \bs y) \wh\xi_1(0, \bs z', \bs y')\der \bs y \der \bs z \der \bs y'  \der \bs z'\\
   &= (2\pi)^{-d(j-1)/2} j^{-d/2} \kappa_1.
\end{align*}
Thus,
$
    \vartheta_j(r) \asymp r^{d(2k-2)}(\sigma\sqrt{\Delta})^{d(1-j)}.
$
For $\vartheta_j'(r)$, due to \eqref{star2} and the choice of $R$ we can write 
\begin{align*}
    \vartheta_j'(r) &\ll r^{d(2k-2)} (\sigma\sqrt{\Delta})^{-dj}  (\sigma\sqrt{\Delta})^{\sqrt{\log(\sigma\sqrt{\Delta})^{-1}}}\ll r^{d(2k-2)} (\sigma\sqrt{\Delta})^{d(1-j)} \ll \vartheta_j(r).
\end{align*} 
Therefore, we obtain,
\begin{align*}
  \ex[\bar{f}_n(0) \bar{f}_n(\Delta)] \approx \frac {\sum\limits_{j=1}^{k} \exp(-j\Delta)  \frac{ \kappa_1}{j!((k-j)!)^2} (2\pi\Delta)^{-d(j-1)/2}j^{-d/2} n^{-j} r^{d(2k-2)}\sigma^{d(1-j)}   } 
    { \sum\limits_{j=1}^{k}\frac{ \kappa_{j}}{j!((k-j)!)^2 } n^{-j} r^{d(2k-j-1)}}. 
\end{align*}
Rearranging the terms conclude the proof.
\end{proof}

\subsection{Covariance of the integrated process in the fast regime}\label{covintfastsubsec}
Before we present the proof of Lemma~\ref{covNoiselemma}
we give some results that concern the asymptotic behavior of the normalization term used to define $\tilde f_n(t)$, viz.
\[M_n= \int_{0}^1\ex[\bar{f}_{n}(0) \bar{f}_{n}(t)]\der t. \] 
Calculating exact asymptotics of this integral is impossible without the explicit characterization of the covariance in the transition regimes, which is out of the scope of this paper. However we can find  upper and lower bounds as given in the following proposition.

\begin{proposition}\label{propintasymp}
   Under the assumptions of Theorem~\ref{white noise thm}
    \begin{align*}
      \left(\frac{r}{\sigma}\right)^{2+\epsilon} \lesssim M_n \lesssim \left(\frac{r}{\sigma}\right)^{2 - \frac{4}{d(k-1)+2}}
    \end{align*}
    for any given $\epsilon>0$.
\end{proposition}

\begin{proof}
    For the lower bound we write
\begin{align*}
    M_n  = \int_{0}^{(r/\sigma)^{2+\epsilon}} \ex[\bar{f}_{n}(0) \bar{f}_{n}(t)]\der t + \int_{(r/\sigma)^{2+\epsilon}}^1\ex[\bar{f}_{n}(0) \bar{f}_{n}(t)]\der t,
\end{align*}
    and due to Lemma~\ref{covOUlemma} we observe
    \rev{\begin{align*}
       \liminf_{n}\inf_{0\leq t\leq (r/\sigma)^{2+\epsilon}} \ex[\bar{f}_{n}(0) \bar{f}_{n}(t)] =1.
    \end{align*}} 
     For the upper bound, we write
    \begin{align}\label{intupperbd}
    M_n = \int_{0}^{(r/\sigma)^{2-\rho}} \ex[\bar{f}_{n}(0) \bar{f}_{n}(t)]\der t + \int_{(r/\sigma)^{2-\rho}}^1\ex[\bar{f}_{n}(0) \bar{f}_{n}(t)]\der t,
    \end{align}
    where $\rho\coloneqq \frac{4}{d(k-1)+2}.$
    Furthermore, due to Lemma~\ref{covNoiselemmabar},
    \begin{align*}
        \ex[\bar{f}_{n}(0) \bar{f}_{n}(t)] &\lesssim \frac {\max_{1\leq j\leq k}  \left(\frac{r}{\sigma}\right)^{\frac{-d(j-1)(2-\rho)}{2}} (n\sigma^d)^{-j} \left(\frac{\sigma}{r}\right)^d} 
    {   (nr^d)^{-k}},
    \end{align*}
    for all $t\geq (r/\sigma)^{2-\rho}$. Note that due to the assumptions of Theorem~\ref{white noise thm}, $r/\sigma \to 0$ and $n\sigma^d\lesssim 1$, the maximum in the numerator is asymptotically achieved for $j=k$. Therefore,
    \begin{align*}
    \ex[\bar{f}_{n}(0) \bar{f}_{n}(t)]&\lesssim   \left(\frac{r}{\sigma}\right)^{\frac{-d(k-1)(2-\rho)}{2}} (n\sigma^d)^{-k} \left(\frac{\sigma}{r}\right)^d (nr^d)^k= \left(\frac{r}{\sigma}\right)^{\rho d(k-1)/2} = \left(\frac{r}{\sigma}\right)^{2-\rho}. 
    \end{align*}
     For $0\leq t\leq(r/\sigma)^{2-\rho}$ , we use the constant bound  $\ex[\bar{f}_{n}(0) \bar{f}_{n}(t)]\leq 1$, and the upper bound on $M_n$ follows from \eqref{intupperbd}.
\end{proof}

\begin{lemma}\label{covNoiseprop}
Under the assumptions of Theorem~\ref{white noise thm}, for any $t,\Delta\geq 0$, 
\begin{align}
    \lim_{n\to\infty} \ex\left[\int_{0}^{t+\Delta}  \tilde f_n(t)\tilde f_n(s) \der s\right] =\frac{1}{2}\ind\{t>0\} + \frac{1}{2}\ind\{\Delta>0\}. \label{covnoiseint1}
\end{align}
\end{lemma}

\begin{proof}

 

To prove \eqref{covnoiseint1}, we first show that the integral and the expectation in \eqref{covnoiseint1} are interchangeable, that is,
\begin{align}
    \ex\left[\int_{0}^{t+\Delta}  \tilde f_n(t)\tilde f_n(s) \der s\right] = \int_{0}^{t+\Delta}  \ex[\tilde f_n(t)\tilde f_n(s)] \der s \label{exintchange}
\end{align}
for all $n$. Note that
\begin{align*}
    \ex[|\tilde f_n(t)\tilde f_n(s)|] &= \frac{\ex\left[\left|\widehat f_n(t)- \ex[\widehat f_n(t)]\right|\left|\widehat f_n(s)- \ex[\widehat f_n(t)]\right| \right] }{2M_n\var[\widehat f_n(t)]} \leq \frac{\ex[\widehat f_n(t)\widehat f_n(s)] + 3\left(\ex[\widehat f_n(t)]\right)^2 }{2M_n\var[\widehat f_n(t)]} ,
\end{align*}
for all $n$ and $t,s \geq 0$. Using stationarity in time,
\begin{align*}
    \ex[\widehat f_n(t)\widehat f_n(s)] \leq \frac{\ex[\widehat f_n(t)^2] + \ex[\widehat f_n(s)^2]}{2} = \ex[\widehat f_n(t)^2],
\end{align*}
which gives
\begin{align}\label{exfntbound}
    \ex[|\tilde f_n(t)\tilde f_n(s)|] 
     \leq \frac{\var[\widehat f_n(t)] + 4\left(\ex[\widehat 
   f_n(t)]\right)^2 }{\var[\widehat f_n(t)] 2M_n}. 
\end{align}
For fixed $n$, $\var[\widehat f_n(t)]$ and $\ex[\widehat f_n(t)]$ were given in terms of $n$ before. Since  $M_n>0$, 
\eqref{exintchange} follows by Fubini.

Next, we focus on the right hand side of \eqref{exintchange}. First assume $t,\Delta>0$, let  $u\coloneqq \frac{t\wedge\Delta}{2}$, and note that  
\begin{align}
     &\int_{0}^{t+\Delta}  \ex[\bar f_n(t)\bar f_n(s)] \der s \nonumber\\
     &=\int_{0}^{t-u}  \ex[\bar f_n(t)\bar f_n(s)] \der s + \int_{t-u}^{t+u}  \ex[\bar f_n(t)\bar f_n(s)] \der s + \int_{t+u}^{t+\Delta}  \ex[\bar f_n(t)\bar f_n(s)] \der s \nonumber \\
    &  =\int_{u}^{t}  \ex[\bar f_n(0)\bar f_n(s)] \der s + 2 \int_{0}^{u}  \ex[\bar f_n(0)\bar f_n(s)] \der s 
    + \int_{u}^{\Delta}  \ex[\bar f_n(0)\bar f_n(s)] \der s \label{intoinfty}.
\end{align}
Since $\sigma /r\to \infty$ and $1\lesssim s $ for all $0< u\leq s\leq t$, we can use Lemma~\ref{covNoiselemmabar} to obtain,  
\begin{align*}  
    \ex[\bar f_n(0)\bar f_n(s)] \leq \sum_{j=1}^k \exp(-js) g(n), 
\end{align*}
for some bounded function $g(n)$ satisfying 
\[ g(n)\lesssim \frac{\max_{1\leq j\leq k} (n\sigma^d)^{-j} \left(\frac{\sigma}{r}\right)^d} {(nr^d)^{-k}} \lesssim \left(r/\sigma\right)^{d(k-1)} \lesssim 1,\]  
from which we conclude that $\ex[\bar f_n(0)\bar f_n(s)]$ is bounded by an integrable function. Therefore, using Reverse Fatou,
\begin{align*}
   \limsup_{n\to\infty}  \int_{u}^{t}  \ex[\bar f_n(0)\bar f_n(s)]\der s \leq \int_{u}^{t} \limsup_{n\to\infty} \ex[\bar f_n(0)\bar f_n(s)]\hspace{0.5em}\der s\lesssim \left(r/\sigma\right)^{d(k-1)}
\end{align*}
for all $0<u<t$. Furthermore, from Proposition~\ref{propintasymp},
\begin{align}
   \frac{\int_{u}^{t}  \ex[\bar f_n(0)\bar f_n(s)] \der s }{M_n}  \ll \left(\frac{r}{\sigma}\right)^{d(k-1)-2-0.5}\ll 1, \label{intoverint}
\end{align}
since $d(k-1)\geq 3$. The same asymptotic upper bound holds for the integral $\int_{u}^{\Delta}  \ex[\bar f_n(0)\allowbreak \bar f_n(s)] \der s$.
Thus, from \eqref{intoinfty},
\begin{align*}
\lim_{n\to\infty} \ex\left[\int_{0}^{t+\Delta}  \tilde f_n(t)\tilde f_n(s) \der s\right] = 
\lim_{n\to\infty} \int_{0}^{t+\Delta}  \ex[\tilde f_n(t)\tilde f_n(s)] \der s = \lim_{n\to\infty} \frac{\int_{0}^{t+\Delta}  \ex[\bar f_n(t)\bar f_n(s)] \der s}{2 M_n}\\= \lim_{n\to\infty} \frac{\int_{0}^{u}  \ex[\bar f_n(0)\bar f_n(s)] \der s}{M_n}= 1 - \lim_{n\to\infty} \frac{\int_{u}^{1}  \ex[\bar f_n(0)\bar f_n(s)] \der s}{M_n} =1
\end{align*}
where we used \eqref{intoverint} in the last step. 

If $t=0$ or $\Delta=0$ but $t+\Delta>0$, write $u\coloneqq  t+\Delta$, and note that 
\begin{align*}
    \lim_{n\to\infty} \ex\left[\int_{0}^{t+\Delta}  \tilde f_n(t)\tilde f_n(s) \der s\right] &= 
    \lim_{n\to\infty} \int_{0}^{u}  \ex[\tilde f_n(t)\tilde f_n(s)] \der s = \lim_{n\to\infty} \frac{\int_{0}^u  \ex[\bar f_n(0)\bar f_n(s)] \der s}{2M_n}
    = \frac{1}{2}
\end{align*}
where we again used Fubini in the first step since the upper bound \eqref{exfntbound} holds for all $t,x\geq 0$. Therefore \eqref{covnoiseint1} follows. 
\end{proof}

\begin{proof}[Proof of Lemma~\ref{covNoiselemma}]
Assume $t_1\leq t_2$ without loss of generality. We write,
\begin{align*}
    \lim_{n\to\infty} \ex\left[\int_{0}^{t_1} \int_{0}^{t_2} \tilde f_n(s_1)\tilde f_n(s_2) \der s_2 \der s_2\right] &= \lim_{n\to\infty} \int_{0}^{t_1} \int_{0}^{t_2} \ex\left[\tilde f_n(s_1)\tilde f_n(s_2)\right] \der s_2 \der s_1  \\
    &=  \int_{0}^{t_1} \lim_{n\to\infty} \int_{0}^{t_1} \ex\left[\tilde f_n(s_1)\tilde f_n(s_2)\right] \der s_2 \der s_1 
    = t_1, 
\end{align*}
where we used Fubini and \eqref{exfntbound} to interchange the expectation and the integral. The rest of the proof follows from the fact that   $\ex[\tilde f_n(s_1)\tilde f_n(s_2)]\geq 0$
for all $s_1,s_2\geq 0$ and from  \eqref{covnoiseint1}.
\end{proof}

\subsection{Finite dimensional distributions}
In this section, we prove that the finite dimensional distributions of the processes of interest to us converge weakly to multivariate Gaussian distributions under the relevant assumptions of Theorems \ref{slow thm}-\ref{white noise thm}.

\begin{lemma}\label{finite}
Suppose that $ \sigma /r \lesssim 1$ and $ n^kr^{d(k-1)} \to \infty$. Then the finite dimensional distributions of $\bar{f}_{n}(t)$ converge to multivariate Gaussian with covariance that depends on the asymptotical behavior of $\sigma/r$.
\end{lemma}
In proving Lemma~\ref{finite} we will use the Cramer-Wold theorem and the marked point process structure we built. Note that, using \eqref{fnhatdef}, the linear combination of $\wh f_n(t)$ across different time samples can be written as 
\begin{align*}
\sum_{i=1}^m \omega_i \widehat f_n(t_i) = \sum_{\mchX\subseteq \wh \eta_{n,T}}  \sum_{i=1}^m \omega_i \wh \xi_r(\mchX (t_i))  \tau_{t_i}(\mchX),
\end{align*}
for some coefficients $\omega_i\neq 0$, $1\leq i\leq m$, and time instances $t_1<\cdots<t_m<T$.
Our proof of Lemma~\ref{finite} uses the normal limit theory developed in \cite{lach1} for functions of finite Wiener chaos expansion on the Poisson space. 

\subsubsection{Wiener Chaos and U-statistics}
Recent developments combining Malliavin calculus with 
Stein method on the Wiener chaos space led to fascinating normal approximation results which eventually found extensive use in several problems of stochastic geometry. The next statement is the essential component in all these results. (See \cite{nualfock} for a proof.) 
\begin{proposition}[Proposition 2.4 in \cite{lach1}]
    Take a Poisson point measure $\eta$ in $\cQ$ with intensity measure $\mu$. Denote the associated \emph{compensated} Poisson measure as $\tilde\eta\coloneqq \eta - \mu$. Every square integrable random variable $G$ with respect to $\eta$ admits a unique chaos decomposition,
   \begin{align}\label{chaosexp}
        G = \ex[G] + \sum_{\ell=1}^\infty I_\ell(g_\ell) ,
    \end{align}
    where each $g_\ell:\cQ^\ell\to \mathbb{R}$ is a square integrable function and $I_\ell(g)$ denotes the multiple Wiener-It\^{o} integral of order $\ell$,  viz.
    \begin{align*}
        I_\ell(g)\coloneqq \int_{\cQ^{\ell}} g(\bs x_{\ell}) \der \tilde\eta^{\otimes \ell}(\bs x_{\ell}).
    \end{align*}
\end{proposition}

\begin{definition}
 $G$ is called a $U$-statistic of order $k$ on the Poisson point process $\eta$ if it satisfies
\begin{align}
    G(\eta) = \sum_{\bs x\in \eta_{\neq}^k} h(\bs x) \label{Ustat} 
\end{align}
for some kernel function $h$, where $\eta_{\neq}^k$ is the set of all $k$-tuples of distinct points in $\eta$. 
\end{definition}

Note that if we assign the following function of the $k$-tuple of marked points  
\begin{align*}
    h(\wh{\bs x}) = \frac{1}{k!} \sum_{i=1}^m \omega_i \wh \xi_r(\wh{\bs x} (t_i))  \tau_{t_i}(\bs x)
\end{align*} 
with the notation 
\begin{align*}
    \wh{\bs x} (t)\coloneqq \left(x_1+Z_{x_1}(t)-Z_{x_1}(B_{x_1}), \ldots, x_k+Z_{x_k}(t)-Z_{x_k}(B_{x_k})\right),  
\end{align*}
then we obtain 
\[G(\wh \eta_{n,T}) =\sum_{i=1}^m \omega_i \widehat f_n(t_i).\]
Therefore, the linear combination $\sum_{i=1}^m \omega_i \widehat f_n(t_i)$, which is of interest to us in the Cramer-Wold theorem, is a special case of a $U$-statistic with a \emph{symmetric} kernel.

The application of the Malliavin-Stein methods to  $U$-statistics was studied in \cite{reitzner2013}. The crucial observation that led to  central limit theorems for  $U$-statistics on Poisson space is the following.
\begin{proposition}[Lemma 3.5 and Theorem 3.6 in \cite{reitzner2013}]\label{UstatWiener}
    Assume the kernel $h$ in \eqref{Ustat} is such that $G$ is square integrable. Then $h$ is also square integrable and $G$ has a finite Wiener chaos expansion. That is, it can be written in the form \eqref{chaosexp} with $g_\ell=0$ for $\ell>k$ and each $g_\ell$ for $1\leq \ell\leq k$ admits the form
    \begin{align*}
        g_\ell(\bs x_\ell) = \binom{k}{\ell}\int_{\cQ^{k-\ell}} h(\bs x_\ell, \bs y_{k-\ell}) \der \mu^{k-\ell}.
    \end{align*} 
\end{proposition}

\subsubsection{CLT for U-statistics} Before we present the CLT (Corollary~\ref{cltcorol}) that we will use to prove Lemma~\ref{finite}, we define \emph{contractions}, constructs that appear in the quantitative normal approximations of $U$-statistics on Poisson processes.
\begin{definition}[Contractions]\label{contractiondef}
Let $\psi:\cQ^i\to\mathbb{R}$, $\phi:\cQ^j\to\mathbb{R}$ be two symmetric functions (for some $i,j\geq 1$) that are square integrable with respect to $\mu^i$ and $\mu^j$ respectively. For every $0\leq \ell\leq m \leq i\wedge j$,   a contraction of $\psi$ and $\phi$ is the function $\psi \starop_m^\ell \phi:\cQ^{i+j-m-\ell}\to \mathbb{R}$ given by
\begin{align*}
    \psi \starop_m^\ell \phi (\bs x_{i-m}, \bs x'_{j-m},\bs y_{m-\ell}) = \int_{\cQ^\ell} \psi(\bs x_{i-m}, \bs y_{m-\ell}, \bs z_\ell)
    \phi(\bs x'_{j-m}, \bs y_{m-\ell}, \bs z_\ell) \prod_{q=1}^m \mu(\der z_q).
\end{align*}

\end{definition}
The normal approximation result we will use is in terms of Wasserstein distance, which we define below.
\begin{definition}[Wasserstein distance]
The Wasserstein distance between two random variables $X$ and $Y$ is defined as
\begin{align*}
    d_W(X,Y)\coloneqq \sup_{f\in \textrm{Lip}_1} \big|\ex[f(X)]-\ex[f(Y)]\big|,
\end{align*}
where $\textrm{Lip}_1$ denotes the set of Lipschitz functions with Lipschitz constant less than or equal to $1$.
\end{definition}

The following central limit theorem is a combination of two previous results and is succinctly expressed as Theorem 2.4 in \cite{lach2}.

\begin{proposition}[Theorem 3.1 in \cite{taqqu} and Theorem 3.5 in  \cite{lach1}]\label{proplach1} Let $\{G_n\}$ be a collection of random variables with finite Wiener chaos expansions, so that
\begin{align*}
    G_n = \ex[G_n] + \sum_{\ell=1}^k I_\ell(g_\ell) 
\end{align*}
    for some fixed $k>0$, where $g_\ell$ implicitly depends on $n$. Assign $\rho_n^2\coloneqq \var[G_n]$ and assume there exists $\rho^2 >0$ such that
    $\lim_{n\to\infty}\rho_n^2 = \rho^2$.
Let $\mcN$ be a standard Gaussian random variable. For every $n$,
\begin{align*}
    d_W(G_n - \ex[G_n], \rho \mcN)\leq \frac{C}{\rho}\left(\max\|g_i\starop_m^\ell g_j\|_{L^2(\mu^{i+j-m-\ell})} + \max_i \|g_i\|^2_{L^4(\mu^i)} \right) \\+ \frac{\sqrt{2/\pi}}{\max\{\rho_n,\rho \}}|\rho_n^2-\rho| 
\end{align*}
for some constant $C$. Note that, $\|\cdot\|_{L^2(\cdot)}$ and $\|\cdot\|_{L^4(\cdot)}$ here denote the second and fourth moments, with respect to a given measure. 
\end{proposition}

Next we state an important inequality regarding contraction kernels given as part of Lemma 2.4 in \cite{dobler}.
\begin{lemma}\label{propdobler}
Let $\psi$ and $\phi$ be as in Definition~\ref{contractiondef}. For all $0\leq \ell\leq m \leq i\wedge j$ we have that 
\begin{align*}
    \|\psi\starop_m^\ell \phi\|_{L^2(\mu^{i+j-m-\ell})}\leq \|\psi\|_{L^4(\mu^i)} \|\phi\|_{L^4(\mu^j)}
\end{align*}
\end{lemma}

Combining Proposition~\ref{proplach1} and Lemma~\ref{propdobler}, we obtain the following corollary for normalized random variables with a  finite chaos expansion. \rev{Normal approximation in this form is particularly suitable for our setting due to its computational form (cf. Theorem 4.6 of \cite{reitzner2013}).}

\begin{corollary}\label{cltcorol}
Let $G_n$ be as in Proposition~\ref{proplach1} and assume $\var[G_n]\to \rho^2>0$. Then 
\begin{align}
\label{416:equn}
    d_W\left(G_n-\ex[G_n], \rho \mcN\right)\lesssim \max_{1\leq i\leq k} \|g_i\|^2_{L^4(\mu^i)}.
\end{align}

\end{corollary}

Now that we have given the necessary background, we are ready to present the proofs of finite dimensional convergence of $\bar f_n(t)$ in different regimes.

\subsubsection{Proofs}\label{proofssec}
We start with the finite-dimensional distributions in the slow and the moderate regimes.
\begin{proof}[Proof of Lemma~\ref{finite}]
The proof follows from calculating the right hand side of \eqref{416:equn} in Corollary~\ref{cltcorol} for the normalized $U$-statistic 
\begin{align}\label{barGndef}
    \bar G_n\coloneqq \frac{\sum_{i=1}^m \omega_i \widehat f_n(t_i) }{\sqrt{\var[\widehat f_n(0)] }} .
\end{align}
The kernel associated with $\bar G_n$ is
\begin{align*}
    \bar h(\wh{\bs x}) = \frac{1}{k!\cdot \sqrt{\var[\widehat f_n(0)] }} \sum_{i=1}^m \omega_i \wh \xi_r(\wh{\bs x} (t_i))  \tau_{t_i}(\wh{\bs x}).
\end{align*}
Using Proposition~\ref{UstatWiener}, the Wiener kernel of degree $\ell$ of interest to us can be written as a function of $n$
\begin{align*}
    \bar g_\ell(\widehat{\bs x}_\ell,n) =  \frac{\binom{k}{\ell}}{k!\cdot \sqrt{\var[\widehat f_n(0)]}} \int_{\cQ^{k-\ell}} \sum_{i=1}^m \omega_i \wh \xi_r(\wh{\bs x}_\ell (t_i), \wh{\bs y}_{k-\ell} (t_i))  \tau_{t_i}(\wh{\bs x}_\ell,\wh{\bs y}_{k-\ell} )\der (\mu_{n,T})^{k-\ell},
\end{align*}
where $\cQ$ now is a shorthand notation for $Q\times\mathbb R^+ \times \mathbb R^+ \times C_{\Rd}[0,T]$ and $\mu_{n,T}$ is the product intensity measure of the marked process $\widehat \eta_{n,T}$. 
Next we will find an asymptotic bound on the square of the fourth moment of $\bar g_\ell$ under $(\mu_{n,T})^{\ell}$ in order to use Corollary~\ref{cltcorol}. First, note that 
\begin{align}
&\int_{\cQ^{\ell}} \left[\int_{\cQ^{k-\ell}} \sum_{i=1}^m \omega_i \wh \xi_r(\wh{\bs x}_\ell (t_i), \wh{\bs y}_{k-\ell} (t_i))  \tau_{t_i}(\wh{\bs x}_\ell,\wh{\bs y}_{k-\ell} ) \der (\mu_{n,T})^{k-\ell}\right]^4 \der (\mu_{n,T})^{\ell}  \nonumber \\
&\leq\max_j|\omega_j|^4 \cdot \int_{\cQ^{\ell}} \left[\int_{\cQ^{k-\ell}} \sum_{i=1}^m \wh \xi_r(\wh{\bs x}_\ell (t_i), \wh{\bs y}_{k-\ell} (t_i))  \der (\mu_{n,T})^{k-\ell}\right]^4 \der (\mu_{n,T})^{\ell} \nonumber
\\
&\begin{alignedat}{2}\label{asympbd}
&\lesssim \int_{\cQ^{\ell}}  \Bigg[ \sum_{i=1}^m \int_{\cQ^{k-\ell}} &&\ind\left\{ \diam(\pi (\wh{\bs x}_\ell(t_i))) \leq \delta r \right\} \\
 &&&  \times  \ind\left\{\max_{1\leq j\leq k-\ell} \left\|\pi (\wh{x}_1(t_i)) - \pi(\wh{y}_j (t_i))\right\| \leq \delta r \right\} \der (\mu_{n,T})^{k-\ell} \Bigg]^4 \der (\mu_{n,T})^{\ell} .
\end{alignedat}
\end{align}
 Now note that for any given $\wh x_1$ and $t_i$,
\begin{align}
 \int_{\cQ^{k-\ell}} \ind\left\{\max_{1\leq j\leq k-\ell} \left\|\pi (\wh{x}_1(t_i)) - \pi(\wh{y}_j (t_i))\right\| \leq \delta r \right\} \der (\mu_{n,T})^{k-\ell} = C (nr^d)^{k-\ell} \label{intindkl}
\end{align}
for some constant $C$, due to the spatial homogeneity (Remark~\ref{stationspacexihat}). Therefore,  
\eqref{asympbd} can be asymptotically  bounded above by
\begin{align*}
\begin{split}
    \int_{\cQ^{\ell}} \Bigg[ (nr^d)^{k-\ell}\sum_{i=1}^m& \ind\left\{ \diam(\pi (\wh{\bs x}_\ell(t_i))) \leq \delta r \right\} \Bigg]^4 \der (\mu_{n,T})^{\ell} \\ &\lesssim m^4 (nr^d)^{4(k-\ell)}  \int_{\cQ^{\ell}} \ind\left\{ \diam(\pi (\wh{\bs x}_\ell(t_1))) \leq \delta r \right\} \der (\mu_{n,T})^{\ell},
    \end{split}
\end{align*}
using the spatial homogeneity again. Using the same techniques as in the calculation of the first moment of $\wh f_n(t)$, the integral on the right hand side can,  asymptotically, be  bounded above by $n(nr^d)^{\ell-1}$, which leads us to conclude that  
\begin{align}\label{maxg4bd}
    \|\bar g_\ell\|^2_{L^4(\mu^\ell)} \lesssim \frac{n^{1/2} (nr^d)^{\frac{4k-3\ell-1}{2}}}{\var[\widehat f_n(0)]}.
\end{align}
On the other hand, we observe that
\begin{align}\label{limvarGn}
    \lim_{n\to\infty}\var[\bar G_n] = \lim_{n\to\infty}\var\left[\sum_{i=1}^m  \frac{\omega_i \widehat f_n(t_i)}{\sqrt{\var[\widehat f_n(0)]}}\right]=  \sum_{i,\ell=1}^m \omega_i \omega_\ell  \lim_{n\to\infty} \cov\left[\bar{f}_n(t_i), \bar{f}_n(t_\ell)\right].
\end{align}
In the slow regime, $\sigma \ll r$, using Lemma~\ref{covOUlemma}, we obtain
\begin{align}\label{varG}
    \lim_{n\to\infty} \var[\bar G_n] =  \sum_{i,\ell=1}^m \omega_i \omega_\ell\sum_{j=1}^k \lambda_j \exp(-j|t_i-t_\ell|) =  \sum_{j=1}^k \lambda_j\bs{\omega}^\top \mathrm{T}^{(j)} \bs{\omega},
\end{align}
for a set of non-negative constants $\lambda_1,\ldots,\lambda_\ell$, with the entries of the matrix $\mathrm{T}^{(j)}$ defined as 
 \[\mathrm{T}^{(j)}_{i\ell} = \exp(-j|t_i-t_\ell|). \]
As this is the covariance matrix of an Ornstein-Uhlenbeck process, $\mathrm{T}^{(j)}$ is a positive definite matrix. Therefore $\bs{\omega}^\top \mathrm{T}^{(j)} \bs{\omega} >0$, for all nonzero $\bs{\omega}\in \mathbb R^m$, and $\rho^2 = \lim_{n\to\infty}\var[\bar G_n]$ is a positive constant for all $\bs{\omega}$. Furthermore, \eqref{varfntasymp} and \eqref{alphaasymp} give that    
\[r^{-d}\sum_{j=1}^k (nr^d)^{2k-j} \lesssim \var[\wh f_n(0)].\]
Using this, together  with Corollary~\ref{cltcorol} and \eqref{maxg4bd},  leads to 
\begin{align}
  d_W\left(\bar G_n - \ex[\bar G_n],\rho \mc N \right)\lesssim \frac{n^{1/2}r^d \max_{1\leq\ell\leq k}(nr^d)^{\frac{4k-3\ell-1}{2}}}{\sum_{j=1}^k (nr^d)^{2k-j}}, \label{dwassbound}
\end{align}
with $\bar G_n$ as defined in \eqref{barGndef}.
If $ nr^d\to 0$, \rev{the $\max$ in the numerator in \eqref{dwassbound} is achieved for $\ell = k$ and the dominant term in the denominator is the one with $j=k$, which gives} 
\begin{align}\label{dwgnbar}
  d_W\left(\bar G_n- \ex[\bar G_n],\rho\mc N \right)\lesssim \frac{n^{1/2}r^d (nr^d)^{\frac{k-1}{2}}}{(nr^d)^k} = \frac{1}{\sqrt{n^k r^{d(k-1)}}} \ll 1,
\end{align}
where we used Assumption~\ref{asymptoassum}. On the other hand, if $ nr^d\to \gamma\in (0,\infty]$, \rev{the $\max$ in the numerator is achieved for $\ell = 1$ and the dominant term in the denominator is the one with $j=1$, which gives}
\begin{align*}
  d_W\left(\bar G_n- \ex[\bar G_n],\rho \mc N \right)\lesssim \frac{n^{1/2}r^d (nr^d)^{2k-2}}{(nr^d)^{2k-1}} = n^{-1/2}.  
\end{align*}

Therefore, the proof of Lemma~\ref{finite} follows in the slow regime. \rev{In the moderate regime, to prove that $\lim_{n\to\infty}\var[\bar G_n]$ is a positive constant, we use Lemma~\ref{covModlemma}, \eqref{definezeta} and \eqref{limvarGn} to write}
\begin{align}\label{varGnorm}
    \lim_{n\to\infty} \var[\bar G_n]  =  \lambda_1\bs{\omega}^\top \mathrm T^{(1)} \bs{\omega}+ \sum_{j=2}^k \frac{\lambda_j}{\kappa_j}\bs{\omega}^\top \left[\tilde{\mathrm T}^{(j)} \circ \mathrm T^{(j)} \right] \bs{\omega},
\end{align}
for some non-negative constants $\tilde\lambda_2,\ldots, \tilde\lambda_k$, and matrices $\tilde{\mathrm T}^{(j)}$. From \eqref{zetajprob}, we have
\begin{align*}
    \mathrm{\tilde T}^{(j)}_{i\ell} = \ex\left[\wh\xi_1(0, \bs z, \bs y) \wh\xi_1(0, (\bs z + \bs w_{\beta|t_i-t_\ell|}) , \bs y')  \right],
\end{align*}
where $\bs w_x$ is a set of $(j-1)$ jointly Gaussian distributed vectors in $\Rd$ with the inverse covariance matrix $\mathrm M^{j,x}$ as  given in \eqref{matrixMdef}, and $\bs z, \bs y, \bs y'$ are iid points in $Q\subset \Rd$ as before. In addition, $\circ$ in \eqref{varGnorm} denotes the Hadamard (entry-wise) product of two matrices. 
\rev{We remark here that, as in the covariance calculations, in the $\beta\to 0^+$ limit, \eqref{varGnorm} reduces to that of the slow regime, as one would expect.} 

Note, also, that $\tilde{\mathrm  T}^{(j)}$ can also be considered as the correlation matrix of the process 
\[\wh\xi_1\left(0, (\bs z + \bs W(\beta t)), \bs y\right), \]
sampled at time points $t_1,\ldots ,t_m$,  where $\{\bs W(\beta t)): t\geq 0\}$ is a stationary Gaussian process in $\mathbb{R}^{d(j-1)}$ with non-degenerate covariance satisfying $\bs W(0)=0$. Consider 
\begin{align*}
    \bs{\omega}^\top \tilde{\mathrm  T}^{(j)}\bs{\omega}^\top =\ex\left[\left(\sum_{\ell=1}^m \omega_\ell  \wh\xi_1\left(0, (\bs z + \bs W(\beta t_\ell)), \bs y\right) \right)^2\right]
\end{align*}
for a nonzero $\bs{\omega}\in \mathbb{R}^m$, and consider the probability
\begin{equation*}
\begin{split}
    &\pr\left[\sum_{\ell=1}^m \omega_\ell  \wh\xi_1\left(0, (\bs z + \bs W(\beta t_\ell)), \bs y\right) \neq 0\right] \\ 
    &\hspace{5em}\geq \pr\Biggl[\wh \xi_1\left(0, \bs z, \bs y\right)>0,\;\; \bigcap_{\ell=2}^m \left\{\wh\xi_1\left(0, (\bs z + \bs W(\beta (t_\ell-t_1))), \bs y\right)=0\right\} \Biggl].
    \end{split}
\end{equation*}
The right hand side is positive due to Remark~\ref{feasxihat} and the fact that $\{\bs W(\beta t)): t\geq 0\}$ has non-degenerate covariance. This leads to the conclusion that $\bs{\omega}^\top \tilde{\mathrm T}^{(j)}\bs{\omega}^\top>0$ for all $\bs{\omega}$, all entries of which are nonzero, and therefore $\tilde{\mathrm T}^{(j)}$ is a positive definite matrix for all $2\leq j\leq k$. 
The product $\tilde{\mathrm T}^{(j)} \circ \mathrm T^{(j)}$ is positive definite as a result of the Schur product theorem \cite[Theorem 7.5.3.]{schur}. Consequently,  \eqref{varGnorm} is strictly positive, and the proof of Lemma~\ref{finite} follows.
\end{proof}

For the fast regime we prove the following lemma.
\begin{lemma}\label{finitefast}
Under the assumptions of Theorem~\ref{white noise thm}, the finite dimensional distributions of the process $\{\int_0^t\tilde{f}_{n}(x)\der x:t\geq 0\}$ converge to multivariate Gaussian.
\end{lemma}
\begin{proof}
\rev{As in the proof of Lemma~\ref{finite} above,
we calculate the right hand side of \eqref{416:equn} in Corollary~\ref{cltcorol}.} Define
\begin{align*}
    \tilde G_n\coloneqq \frac{\sum_{i=1}^m \omega_i \int_0^{t_i}\widehat {f}_{n}(u)\der u }{\sqrt{2\var[\widehat f_n(0)] M_n }} .
\end{align*}
The $U$-statistic kernel associated with $\tilde G_n$ is then
\begin{align*}
    \tilde h(\wh{\bs x}) = \frac{1}{k!\cdot \sqrt{2M_n\var[\widehat f_n(0)] }} \sum_{i=1}^m \omega_i \int_0^{t_i}\wh \xi_r(\wh{\bs x} (u))  \tau_{u} (\wh{\bs x}) \der u,
\end{align*}
and the Wiener kernel of degree $\ell$ is 
\begin{align*}
    \tilde g_\ell(\widehat{\bs x}_\ell,n) =  \frac{\binom{k}{\ell}}{k!\cdot \sqrt{2M_n\var[\widehat f_n(0)]}} \int_{\cQ^{k-\ell}} \sum_{i=1}^m \omega_i \int_0^{t_i} &\wh \xi_r(\wh{\bs x}_\ell (u), \wh{\bs y}_{k-\ell} (u))  \\&\times\tau_{u}(\wh{\bs x}_\ell,\wh{\bs y}_{k-\ell} )\der u \; \der (\mu_{n,T})^{k-\ell}.
\end{align*}
Accordingly,
\begin{align}
&\int_{\cQ^{\ell}} \left[\int_{\cQ^{k-\ell}} \sum_{i=1}^m \omega_i \int_{0}^{t_i}\wh \xi_r(\wh{\bs x}_\ell (u), \wh{\bs y}_{k-\ell} (u))  \tau_{u}(\wh{\bs x}_\ell,\wh{\bs y}_{k-\ell} ) \der u \cdot\der (\mu_{n,T})^{k-\ell}\right]^4 \der (\mu_{n,T})^{\ell}  \nonumber\\
&\lesssim \int_{\cQ^{\ell}} \left[\int_{\cQ^{k-\ell}} \int_0^T \wh \xi_r(\wh{\bs x}_\ell (u), \wh{\bs y}_{k-\ell} (u)) \der u \cdot \der (\mu_{n,T})^{k-\ell}\right]^4 \der (\mu_{n,T})^{\ell} \nonumber\\
&\begin{alignedat}{2} \nonumber 
&\lesssim \int_{\cQ^{\ell}}   \int_0^T \Bigg[ \int_{\cQ^{k-\ell}} &&\ind\left\{ \diam(\pi (\wh{\bs x}_\ell(u))) \leq \delta r \right\} \\
   &&&\hspace{-1em}\times \ind\left\{\max_{1\leq j\leq k-\ell} \left\|\pi (\wh{x}_1(u)) - \pi(\wh{y}_j (u))\right\| \leq \delta r \right\} \der (\mu_{n,T})^{k-\ell} \Bigg]^4 \der u \cdot \der (\mu_{n,T})^{\ell}, 
\end{alignedat}
\end{align}
using Jensen's inequality for the final inequality.
Using \eqref{intindkl} for $t_i=u$ we obtain that 
the above 
is  bounded  by
\begin{align*}
    C (nr^d)^{4(k-\ell)}  \int_{\cQ^{\ell}} \int_0^T \ind\left\{ \diam(\pi (\wh{\bs x}_\ell(u))) \leq \delta r \right\} \der u\cdot \der (\mu_{n,T})^{\ell}
\end{align*}
for some constant $C$. Therefore, arguing as in the paragraph in the proof of Lemma~\ref{finite} leading to \eqref{maxg4bd}, we find
\begin{align*}
    \|\tilde g_\ell\|^2_{L^4(\mu^\ell)} \lesssim \frac{n^{1/2} (nr^d)^{\frac{4k-3\ell-1}{2}}}{M_n \var[\widehat f_n(0)]}.
\end{align*}
Thus, similar to \eqref{dwgnbar}, since $ nr^d\to 0$
\begin{align*}
  d_W\left(\tilde G_n- \ex[\tilde G_n],\tilde\rho\mc N \right)\lesssim  \frac{1}{M_n\sqrt{n^k r^{d(k-1)}}}\ll \frac{\sigma^{2+\epsilon}}{r^{2+\epsilon}\sqrt{n^k r^{d(k-1)}}}  
\end{align*}
for any $\epsilon>0$, due to Proposition~\ref{propintasymp}. The convergence follows from the assumption of Theorem~\ref{white noise thm} that $\sigma /r\ll (n^kr^{d(k-1)})^{\frac{1}{4}-\epsilon'}$ for some $\epsilon'$.
Also,
\begin{align*}
    \tilde\rho^2 = \lim_{n\to\infty}\var[\tilde G_n] &= \lim_{n\to\infty}\var\left[\sum_{i=1}^m \sum_{\ell=1}^m \omega_i  \int_0^{t_i}\tilde f_n(u)\der u\right] = \sum_{i=1}^m  \omega_i \omega_\ell (t_i\wedge t_\ell) >0,
\end{align*}
due to Lemma~\ref{covNoiselemma}. This concludes the proof.

\end{proof}

\begin{proof}[Proof of Theorem \ref{slow thm}] 
Follows from Lemma~\ref{covOUlemma} and Lemma~\ref{finite}.
\end{proof}

\begin{proof}[Proof of Theorem \ref{ModerateThm}]
Follows from Lemma~\ref{covModlemma} and  Lemma~\ref{finite}.
\end{proof}

\begin{proof}[Proof of Theorem \ref{white noise thm}]
Follows from Lemma~\ref{covNoiselemma} and  Lemma~\ref{finitefast}.
\end{proof}

\bibliography{references}      

\end{document}